\documentclass[11 pt]{amsart}

\usepackage{graphicx, amsmath, fullpage, amssymb, amsthm, amsfonts, color}
\usepackage{enumerate}
\newtheorem{theorem}{Theorem}[section]
\newtheorem{lemma}[theorem]{Lemma}

\theoremstyle{definition}
\newtheorem{definition}[theorem]{Definition}
\newtheorem{example}[theorem]{Example}

\newtheorem{rmk}[theorem]{Remark}
\newtheorem{cor}[theorem]{Corollary}
\numberwithin{equation}{section}
\usepackage[colorlinks,
            linkcolor=blue,
            anchorcolor=blue,
            citecolor=blue
            ]{hyperref}
\begin{document}

\title{A Graphon Approach to Limiting Spectral Distributions  of Wigner-type Matrices}

\author{Yizhe Zhu}
\address{Department of Mathematics, University of Washington, Seattle, WA 98195, USA}
\curraddr{Department of Mathematics, University of California, San Diego, La Jolla, CA 92093, USA}

\email{yiz084@ucsd.edu}
\thanks{Y.Z. is partially supported by NSF DMS-1712630.}
%
\subjclass[2000]{Primary 05C80, 15A52; Secondary 60C05, 90B15}
\date{\today}
\keywords{graphon; homomorphism density; spectral distribution; inhomogeneous random graph; Wigner-type matrix}

\begin{abstract}
We present a new approach, based on graphon theory, to finding the limiting spectral distributions of general Wigner-type matrices.  This approach determines the moments of the limiting measures and the equations of their Stieltjes transforms explicitly with weaker assumptions on the convergence of variance profiles than previous results.  As applications, we give a new proof of the semicircle law for generalized Wigner matrices and determine the limiting spectral distributions for three sparse inhomogeneous random graph models with sparsity $\omega(1/n)$: inhomogeneous random graphs with roughly equal expected degrees, $W$-random graphs and stochastic block models with  a growing number of blocks. Furthermore, we show  our theorems can be applied to random Gram matrices with a variance profile for which we can find the limiting spectral distributions under weaker assumptions than previous results.
\end{abstract}

\maketitle

%

\section{Introduction} 

\subsection{Eigenvalue Statistics of Random Matrices}

Random matrix theory is a central topic in probability and statistical  physics with many connections to various areas such as combinatorics, numerical analysis, statistics, and theoretical computer science. One of the primary goals of random matrix theory is to study the limiting laws for eigenvalues of $(n\times n)$ Hermitian random matrices as $n\to\infty$. 

Classically, a Wigner matrix is a Hermitian random matrix whose entries are i.i.d. random variables up to the symmetry constraint, and have zero expectation and  variance 1. 
As has been known since Wigner's seminal paper \cite{wigner1955characteristic} in various formats,  for Wigner matrices, the empirical spectral distribution converges almost surely to the semicircle law. The i.i.d. requirement and the constant variance condition are not essential for proving the semicircle law, as can be seen from the fact that generalized Wigner matrices, whose entries have different variances but each column of the variance profile is stochastic, turned out to obey the semicircle law \cite{anderson2006clt,erdHos2012bulk,girko1994necessary}, under various conditions as well. Beyond the semicircle law, the Wigner matrices exhibit universality \cite{erdHos2010bulk,tao2011random}, a phenomenon that has been recently shown to hold for other models, including generalized Wigner matrices   \cite{erdHos2012bulk},  adjacency matrices of Erd\H{o}s-R\'{e}nyi random graphs  \cite{erdHos2012spectral,erdHos2013spectral,tran2013sparse,huang2015bulk}  and general Wigner-type matrices \cite{ajanki2015universality}.

A slightly different direction of research is to investigate structured random matrix models whose limiting spectral distribution is not the semicircle law. One such example is random block matrices, whose  limiting spectral distribution has been found in \cite{shlyakhtenko1996random,far2008slow} using free probability. Ding \cite{ding2014spectral} used moment methods to derive the limiting spectral distribution of random block matrices  for a fixed number of blocks (a claim in
\cite{ding2014spectral} that the method extends to the growing number of blocks case is
unfortunately incorrect). Recently Alt et al. \cite{alt2017location} provided a unified way to study the global law for a general class of non-Hermitian random block matrices including Wigner-type matrices.

\subsection{Graphons and Convergence of Graph Sequences}
Understanding large networks is a fundamental problem in modern graph theory and to properly define a limit object, an important issue is to have  good definitions of convergence for graph sequences. Graphons, introduced in 2006 by Lov{\'a}sz and Szegedy \cite{lovasz2006limits} as limits of dense graph sequences, aim to provide a  solution to this question. Roughly speaking, the set of finite graphs endowed with the cut metric (See Definition \ref{cutmetric}) gives rise to a metric space, and the completion of this space is the space of graphons. These objects may be realized as symmetric, Lebesgue measurable functions from $[0, 1]^2$ to $\mathbb R$. They also characterize the convergence of graph sequences based on graph homomorphism densities
\cite{borgs2008convergent, borgs2012convergent}. Recently, graphon theory has been generalized for sparse graph sequences  \cite{borgs2014p,borgs2018,frenkel2018convergence,kunszenti2019measures}. 

The most relevant results for our endeavor are the connections between two types of convergences: left convergence in the sense of homomorphism densities and convergence in cut metric. 
In our approach, for the general Wigner-type matrices, we will regard the variance profile  matrices $S_n$ as a graphon sequence. The convergence of empirical spectral distributions is connected to the convergence of this graphon sequence associated with $S_n$ in either left convergence sense or in cut metric.

  \subsection{Random Graph Models}
 One of the most basic models for random graphs is the Erd\H{o}s-R\'{e}nyi random graph. The scaled adjacency matrix $\frac{A_n}{\sqrt {np}}$ of Erd\H{o}s-R\'{e}nyi random graph $\mathcal{G}(n,p)$ has the semicircle law as limiting spectral distribution \cite{ding2010spectral,tran2013sparse} when $np\to\infty$. 
  
  Random graphs generated from an inhomogeneous Erd\H{o}s-R\'{e}nyi model $\mathcal G(n, (p_{ij}))$, where edges exist independently with given probabilities $p_{ij}$ is a generalization of the classical Erd\H{o}s-R\'{e}nyi model $\mathcal{G}(n,p)$. Recently, there are some results on the largest eigenvalue \cite{benaych2017largest,benaych2017spectral} and the spectrum of the Laplacian matrices \cite{chakrabarty2018spectra} of  inhomogeneous Erd\H{o}s-R\'{e}nyi model random graphs.  Many popular graph models arise as special cases of $\mathcal G(n, (p_{ij}))$ such as random graphs with given expected degrees \cite{chung2003spectra}, stochastic block models \cite{holland1983stochastic}, and $W$-random graphs \cite{lovasz2006limits,borgs2014p}. 
  
  The stochastic block model (SBM) is a random graph model with planted clusters. It is widely used as a canonical model to study clustering and community detection in network and data sciences \cite{abbe2018community}. Here one assumes that a random  graph was generated by first partitioning vertices  into unknown $d$ groups, and then connecting two vertices with a probability that depends on their assigned groups. Specifically, suppose we have a partition of $[n] = V_1 \cup V_2
\cup \ldots \cup V_d$ for some integer $d$, and that $|V_i|=n_i$ for $i = 1, \ldots, d$. Suppose that for any pair
$(k,l) \in [d] \times [d]$  there is a $p_{kl} \in [0,1]$ such that for
any $i \in V_k$, $j \in V_l$, 
\[
a_{ij} = \left \{ \begin{array}{ll} 1, 
    & \mbox{with probability } p_{kl}, \\
0, & \mbox{otherwise}. \end{array} \right . 
\]
Also, if $k=l$, there is a $p_{kk}$ such that $a_{ii}=0$ for $i\in V_k$ and for any $i\not= j, i,j \in V_k$, 
\[
a_{ij} = \left \{ \begin{array}{ll}  
                    1, & \mbox{with probability $p_{kk}$,} \\
                    0, & \mbox{otherwise}. \end{array} \right .
\]
The task for community detection is to find the unknown partition of a random graph sampled from the SBM. In this paper, we will consider the limiting spectral distribution of the adjacency matrix of an SBM. Since permuting the adjacency matrix does not change its spectrum, we may assume its adjacency matrix has a block structure by a proper permutation.

 As the number of vertices grows, the network might not be well described by a stochastic block model with a fixed number of blocks. Instead, we might consider the case where the number of blocks grows as well \cite{choi2012stochastic} (see Section \ref{sSBM}). A different model that generates  nonparametric random graphs is called $W$-random graphs and is achieved by sampling points uniformly from a graphon $W$. We will define a sparse version of $W$-random graphs in Section \ref{wrandomgraph} for which one can obtain a limiting spectral distribution when the sparsity $\rho_n=\omega(1/n)$. 
 
 For  inhomogeneous random graphs with bounded expected degree introduced by Bollob\'{a}s, Janson and Riordan \cite{bollobas2007phase}, their graphon limits will be $0$ and our main result will not cover this regime. This is because the graphon limit is only suitable for  graph sequences with unbounded degrees. Instead, the spectrum of random graphs  with bounded expected degrees was studied in \cite{bordenave2010resolvent} by local weak convergence \cite{benjamini2001recurrence,aldous2004objective}, a graph limit theory for graph sequences with bounded degrees.

 \subsection{Random Gram Matrices} Let $X$ be a $m\times n$ random matrix with independent, centered entries with unit variance, where $\frac{m}{n}$ converges to some positive constant as $n\to\infty$. It is known that the empirical spectral distribution converges to the Mar{\v{c}}enko-Pastur law \cite{marvcenko1967distribution}. However, some applications in wireless communication require understanding the spectrum of $\frac{1}{n}XX^*$ where $X$ has a variance profile \cite{hachem2008clt,couillet2011random}. Such matrices are called \textit{random Gram matrices}. The limiting spectral distribution of a random Gram matrix  with non-centered diagonal entries and a variance profile was obtained in \cite{hachem2006empirical} under the assumptions that the $(4+\varepsilon)$-th moments of entries in $X$ are bounded and the variance profile comes from a continuous function. The local law and singularities of the density of states of random Gram matrices were analyzed in \cite{alt2017local,alt2017}.
 
 We use the symmetrization trick to connect the eigenvalues of $\frac{1}{n}X X^*$ to eigenvalues of a Hermitian matrix $H:=\begin{bmatrix}
	0 & X\\
	X^* &0
\end{bmatrix}$. As a corollary from our main theorem in Section \ref{wignertype}, when $\mathbb E X=0$, we obtain the moments and Stieltjes transforms of the limiting spectral distributions under weaker assumptions than \cite{hachem2006empirical}. In particular, we only need entries in $X$ to have finite second moments, and the variance profile of $H_n$  converges in terms of homomorphism densities.

 \subsection{Contributions of this Paper}
 We obtained a  formula to compute the moments of limiting spectral distributions of general Wigner-type matrices from graph homomorphism densities, and we derived  quadratic vector equations as in \cite{ajanki2015quadratic}  from this formula.
 
  Previous approaches to the problem require the variance profiles to converge to a function whose set of discontinuities has measure zero \cite{shlyakhtenko1996random,anderson2006clt,hachem2006empirical}, we make no such requirement here.
 The method in  \cite{shlyakhtenko1996random} is based on free probability theory,  and it is assumed  that all entries of the matrix are Gaussian, while our Theorem \ref{main} and Theorem \ref{main2} work for non-i.i.d. entries with general distributions. Especially, we cover a variety of sparse matrix models (see Section \ref{generalized}-\ref{sSBM}).  The argument in \cite{anderson2006clt} is based on a sophisticated moment method for band matrix models, and our moment method proof based on graphon theory is much simpler and can be applied to many different models including random Gram matrices.  For random Gram matrices, in \cite{hachem2006empirical}, it is assumed that all entries have $(4+\varepsilon)$ moments and the variance profile is continuous. The continuity assumption is used to show the Stieltjes transform of the empirical measure converges to the Stieltjes transform of the limiting measure. We remove the technical higher moments and the continuity assumptions since our combinatorial approach requires less regularity.

 All the three previous results above assume the limiting variance profile exists and is continuous. This assumption is used to have an error control under $L^{\infty}$-norm  between the $n$-step variance profile and the limiting variance profile, which will guarantee that either the moments of the empirical measure converge or the  Stieltjes transform the empirical measure converges.  
 However, this $L^{\infty}$-convergence is  only a stronger sufficient condition compared to our condition in Theorem \ref{main} and Theorem \ref{main2}.  The key observation in our approach  is that permuting a random matrix does not change its spectrum, but the continuity of the variance is destroyed. The cut metric  in the graphon theory is a suitable tool to exploit the permutation invariant property of the spectrum (see Theorem \ref{main2}).

 Moreover, we realize that to make the moments of the empirical measure converge, we don't need to assume the moments of the limiting measure is an integral in terms of the limiting variance profile. All we need is the convergence of homomorphism density from trees. We show two examples in Section \ref{generalized} where we don't have a limiting variance profile but the moments of the empirical measure still converge: generalized Wigner matrices and inhomogeneous random graphs with roughly equal expected degrees. 

  Besides, if the limiting distribution is not the semicircle law,  previous results only  implicitly characterize the Stieltjes transform of the limiting measure by the quadratic vector equations (see \eqref{QVE1}, \eqref{QVE2}), which are not easy to solve. Our combinatorial approach explicitly determines  the moments of the limiting distributions in terms of sums of graphon integrals. Our convergence condition (see Theorem \ref{main} (1)) is the weakest so far for the existence of limiting spectral distributions and covers a variety of models like  generalized Wigner matrices, adjacency matrices of sparse stochastic block models with a growing number of blocks, and random Gram matrices.

 The organization of this paper is as follows: In Section \ref{pre}, we introduce  definitions and facts that will be used  in our proofs.  In Section \ref{wignertype}, we  state and prove the main theorems for general Wigner-type matrices  and then specialize our results to different models in Section \ref{generalized}-\ref{sSBM}. In Section \ref{gram}, we extend our results to random Gram matrices with a variance profile.
 \section{Preliminary}\label{pre}
 
\subsection{Random Matrix Theory}
We recall some basic definitions in random matrix theory.
For any $n\times n$ Hermitian matrix $A$ with eigenvalues $\lambda_1,\dots,\lambda_n$, the \textit{empirical spectral distribution} (ESD) of $A$ is defined by
\begin{align*}
F^A(x):=\frac{1}{n}\sum_{i=1}^n\mathbf{1}_{\{\lambda_i\leq x\}}.	
\end{align*}
Our main task in this paper is to investigate the convergence of the sequence of empirical spectral distribution $\{F^{A_n}\}$ to the limiting spectral distribution for a given sequence of structured random matrices. A useful tool to study the convergence of measure is the Stieltjes transform. 

Let $\mu$ be a probability measure on $\mathbb R$. The \textit{Stieltjes transform} of $\mu$ is a function $s(z)$ defined	on the upper half plane $\mathbb C^+$ by the formula:
\begin{align*}
s(z)=\int_{\mathbb R}\frac{1}{z-x}d\mu(x), \quad z\in\mathbb C^+.	
\end{align*}
Suppose that $\mu$ is compactly supported, and denote $  r:=\sup\{|t| \mid t\in \text{supp}(\mu)\}.$ We then have a power series expansion
\begin{align}\label{pse}
s(z)=\sum_{k=0}^{\infty}\frac{\beta_k}{z^{k+1}}, \quad |z|\geq r,	
\end{align}
where $\beta_k:=\int_{\mathbb R} x^k d\mu(x)$ is the $k$-th moment of $\mu$ for $k\geq 0$.

We recall some combinatorial objects related to random matrix theory.
\begin{definition}
The \textit{rooted planar tree} is a planar graph with no cycles, with one distinguished vertex as a root, and with a choice of ordering at each vertex. The ordering defines a way to explore the tree starting at the root.
	\textit{Depth-first search} 
is an algorithm for traversing rooted planar trees. One starts at the root  and explores as far as possible along each branch before backtracking. An enumeration of the vertices of a tree is said to have \textit{depth-first search order} if it is the output of the depth-first search. 
\end{definition}
 The Dyck paths of  length $2k$ are bijective to rooted planar trees of $k+1$ vertices by the depth-first search (see Lemma 2.1.6 in \cite{anderson2010introduction}).  Hence the number of rooted planar trees with $k+1$ vertices is the $k$-th Catalan number $   C_k:=\frac{1}{k+1}{{2k}\choose{k}}.$
\subsection{Graphon Theory}
We introduce definitions from graphon theory. For more details, see \cite{lovasz2012large}.
\begin{definition}
A \textit{graphon} is a symmetric, integrable function $W:[0,1]^2\to \mathbb R$.	
\end{definition}
Here symmetric means $W(x,y)=W(y,x)$ for all $x,y\in[0,1]$.  Every weighted graph $G$ has an associated graphon $W^G$ constructed as follows. First divide the interval $[0,1]$ into intervals $I_1,\dots, I_{|V(G)|}$ of length $\frac{1}{|V(G)|}$, then give the edge weight  $\beta_{ij}$ on $I_i\times I_j$, for all $i,j\in V(G)$. In this way, every finite weighted graph gives rise to a graphon ({see Figure \ref{graphon}}).
\begin{figure}
\includegraphics[width=8 cm]{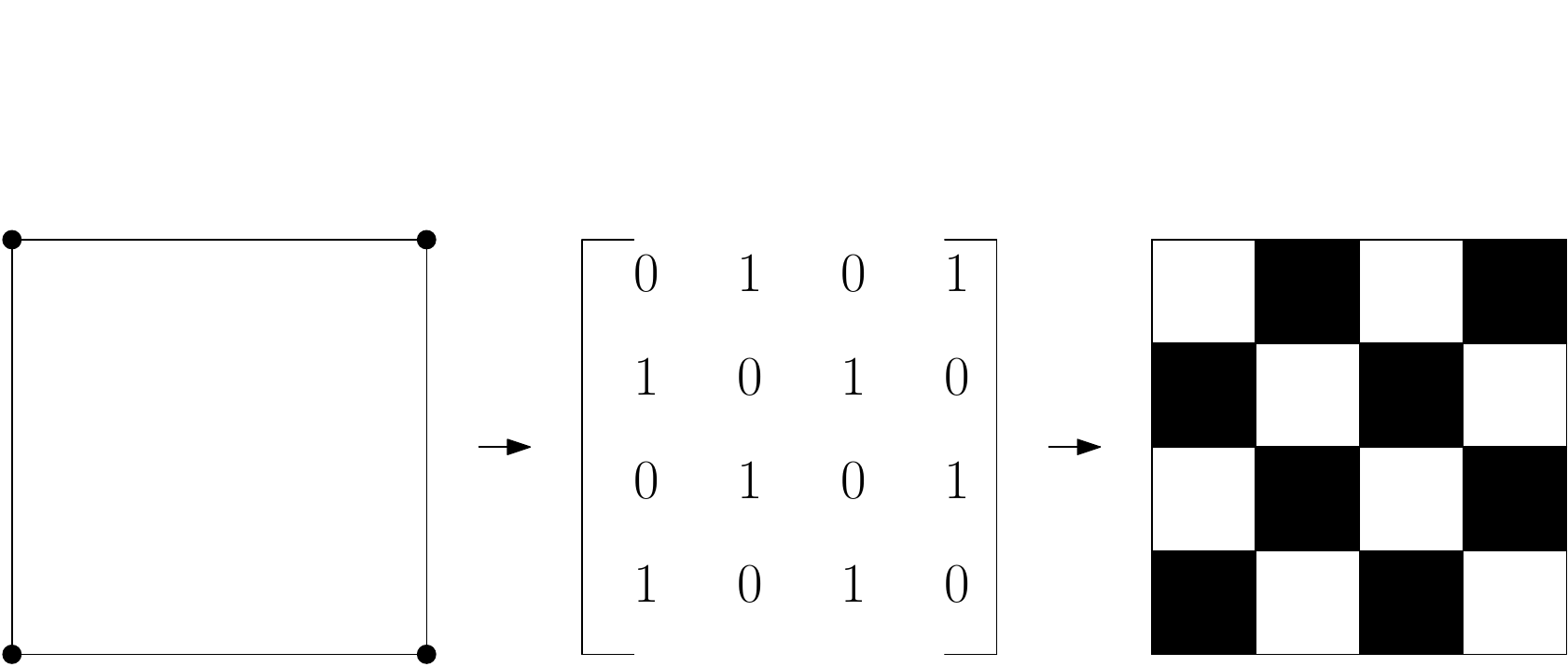}
\caption{Graphon representation of a graph}
\label{graphon}
\end{figure}

The most important metric on the space of graphons is the cut metric. The space that contains all graphons taking values in $[0,1]$ endowed with the cut metric is a compact metric space. 

\begin{definition}\label{cutmetric}
For a graphon $W:[0,1]^2\to\mathbb R$, the \textit{cut norm} is defined by
\begin{align*}
	\|W\|_{\Box}:=\sup_{S,T\subseteq [0,1]}\left|\int_{S\times T} W(x,y)dx dy\right|,
\end{align*}
	where $S$, $T$ range over all measurable subsets of $[0,1]$. Given two graphons $W, W': [0,1]^2\to\mathbb R$, define 
$
	d_{\Box}(W,W'):=\|W-W'\|_{\Box}
$
and the \textit{cut metric} $\delta_{\Box}$ is defined by
\begin{align}
\delta_{\Box}(W,W'):=\inf_{\sigma} d_{\Box}(W^{\sigma},W'),	 \notag
\end{align}
where $\sigma$ ranges over all measure-preserving bijections $[0,1]\to [0,1]$ and 
$W^{\sigma}(x,y):=W(\sigma(x),\sigma(y))$.

\end{definition}

Using the cut metric, we can compare two graphs with different sizes and  measure their similarity, which defines a type of convergence of graph sequences whose limiting object is the graphon we introduced. Another way of defining the convergence of graphs is to consider graph homomorphisms.
  
 \begin{definition}
 	For any graphon $W$ and multigraph $F=(V,E)$ (without loops), define the \textit{homomorphism density} from $F$ to $W$ as 
\begin{align}
t(F,W):=\int_{[0,1]^{|V|}} \prod_{ij\in E}W(x_i,x_j) \prod_{i\in V}dx_i.	\notag
\end{align}
 \end{definition} 

One may define \textit{homomorphism density from partially labeled graphs to graphons}, as follows.
\begin{definition}
	Let $F=(V,E)$ be a $k$-labeled multigraph. Let $V_0=V\setminus [k]$ be the set of unlabeled vertices. For any graphon $W$, and $x_1,\dots, x_k\in [0,1]$, define
\begin{align}
t_{x_1,\dots, x_k}(F,W):=\int_{x\in [0,1]^{|V_0|}}\prod_{ij\in E}W(x_i,x_j) \prod_{i\in V_0}dx_i.
\end{align}
This is a function of $x_1,\dots,x_k$.
\end{definition} 

It is natural to think  two graphons $W$ and $W'$ are similar if they have similar homomorphism densities from any finite graph $G$. This leads to the following definition of left convergence.

\begin{definition}
	Let $W_n$ be a sequence of graphons. We say  $W_n$ is \textit{convergent from the left} if $t(F,W_n)$ converges for any finite simple  (no loops, no multi-edges, no directions) graph $F$. 
\end{definition}
The importance of homomorphism densities is that they characterize convergence under the cut metric.
Let $\mathcal W_0$ be the set of all graphons such that $0\leq W\leq 1$.
The following is a characterization of convergence in the space $\mathcal W_0$, known as Theorem 11.5 in \cite{lovasz2012large}.

\begin{theorem} \label{graphonconvergence}
Let $\{W_n\}$ be a sequence of graphons in $\mathcal W_0$ and let $W\in \mathcal W_0$.	Then  $t(F,W_n)\to t(F,W)$ for all finite simple graphs if and only if $\delta_{\Box}(W_n,W)\to 0$.
\end{theorem}

\section{Main Results for General Wigner-type Matrices}\label{wignertype}
\subsection{Set-up and Main Results}

Let $A_n$ be a Hermitian random matrix whose entries above and on the diagonal of $A_n$ are independent. Assume a \textit{general Wigner-type matrix} $A_n$ with a variance profile matrix $S_n$ satisfies the following conditions:
\begin{enumerate}
\item $\mathbb Ea_{ij}=0,\mathbb E|a_{ij}|^2=s_{ij}$. 
\item (Lindeberg's condition) for any constant $\eta>0$,\label{eta}
\begin{align}\lim_{n\to\infty} \frac{1}{n^2}\sum_{1\leq,i,j\leq n}\mathbb E[	|a_{ij}|^2\mathbf{1}(|a_{ij}|\geq \eta\sqrt n)]=0.\label{lind}	
\end{align}
\item $  \sup_{ij}s_{ij}\leq C $ for some  constant $C\geq 0$.
\end{enumerate}
\begin{rmk}
If we assume entries of $A_n$ are of the form $a_{ij}=s_{ij}\xi_{ij}$ where the $\xi_{ij}$'s have mean 0, variance 1 and are i.i.d. up to symmetry, then the Lindeberg's condition \eqref{lind} holds by the Dominated Convergence Theorem.
\end{rmk}

To begin with, we associate a graphon $W_n$ to the matrix $S_n$ in the following way. Consider $S_n$ as the adjacency matrix of a weighted graph $G_n$ on $[n]$ such that the weight of the edge $(i,j)$ is $s_{ij}$, then $W_n$ is defined as the corresponding graphon to $G_n$. We say $W_n$ is a \textit{graphon representation} of $S_n$. We define $M_n:=\frac{1}{\sqrt n}A_n$ and denote all rooted planar tree with $k+1$ vertices as $T_{j}^{k+1}, 1\leq j\leq C_k$. Now we are ready to state our main results for the limiting spectral distributions of general Wigner-type matrices.

\begin{theorem} Let $A_n$ be a general Wigner-type matrix   and $W_n$ be the corresponding graphon of $S_n$. The following holds:\label{main}
\begin{enumerate}
\item If for any finite tree $T$, $t(T,W_n)$ converges as $n\to\infty$, the empirical spectral distribution of $M_n$ converges almost surely to a probability measure $\mu$ such that for $k\geq 0$, 
\begin{align*}
 \int x^{2k} d\mu &=\sum_{j=1}^{C_k}\lim_{n\to\infty}t(T_j^{k+1},W_n),	\quad 
 \int x^{2k+1} d\mu =0.
 \end{align*} 

\item If $\delta_{\Box}(W_{n},W)\to 0$ for some graphon $W$  as $n\to\infty$, then  for all $k\geq 0$, 
\begin{align*}
	\int x^{2k} d\mu &=\sum_{j=1}^{C_k}t(T_j^{k+1},W), \quad 
 \int x^{2k+1} d\mu =0.
\end{align*} 
\end{enumerate}
\end{theorem}

\begin{rmk}
Similar moment formulas appear in the study of traffic distributions in free probability theory \cite{male2011traffic,male2014uniform}. 	
\end{rmk}
 
Using the connection between the moments of the limiting spectral distribution and its Stieltjes transform described in \eqref{pse}, we can derive the  equations for the Stieltjes transform of the limiting measure  by the following theorem. 

\begin{theorem}\label{main2}
Let $A_n$ be a general Wigner-type matrix and $W_n$ be the corresponding graphon of $S_n$. If $\delta_{\Box}(W_n,W)\to 0$ for some graphon $W$, then the empirical spectral distribution of $M_n:=\frac{A_n}{\sqrt n}$ converges almost surely to a probability measure $\mu$ whose Stieltjes transform $s(z)$ is an analytic solution defined on $ \mathbb C^+$ by  the following equations:
\begin{align}
s(z)&=\int_{0}^{1} a(z,x) dx,\label{QVE1} \\
a(z,x)^{-1}&=z-\int_{0}^{1} W(x,y)a(z,y) dy, \quad x\in [0,1],	\label{QVE2}
\end{align}
where $a(z,x)$ is the unique analytic solution of \eqref{QVE2} defined on $\mathbb C^+\times [0,1]$. 

Moreover, for $|z|>2\|W \|^{1/2}_{\infty}$,
\begin{align}
a(z,x) &=\sum_{k=0}^{\infty}\frac{\beta_{2k}(x)}{z^{2k+1}}, \quad 
\beta_{2k}(x):=\sum_{j=1}^{C_k}t_{x}(T_{j}^{k+1},W),\label{DEF2}\\
\text{where}\quad  t_{x_1}(T_{j}^{k+1},W):&=\int_{[0,1]^{k}}\prod_{uv\in E(T_{j}^{k+1})} W(x_u,x_v) \prod_{i=2}^{k+1} dx_i. 	\label{DEF}
 \end{align}

\end{theorem}
\noindent
\begin{rmk} 
In \eqref{DEF}, $t_{x_1}(T_{j}^{k+1},W)$ is a function of $x_1$, and  in \eqref{DEF2} $t_{x}(T_{j}^{k+1},W)$ is the function evaluated at $x_1=x$. 
\end{rmk}

Theorem \ref{main2} holds under a stronger condition compared to Theorem \ref{main}. We provide two examples in Section \ref{generalized} to show that it's possible to have tree densities converge but the empirical graphon does not converge under the cut metric. We show that  the limiting spectral distribution can still exist.  However, to have the equations \eqref{QVE1} and \eqref{QVE2}, we need a well-defined measurable function $W$ that $W_n$ converges to, therefore we need the condition of graphon convergence under the cut metric.

 \eqref{QVE1} and \eqref{QVE2} have been known as \textit{quadratic vector equations} in \cite{ajanki2015quadratic,ajanki2016singularities}, where the properties of the solution are discussed under more assumptions on  variance profiles to prove local law and universality. A similar expansion as \eqref{DEF2} and \eqref{DEF} has been derived  in \cite{erdHos2018bounds}. The central role of \eqref{QVE2} in the context of random matrices has been recognized by many authors, see \cite{girko2001theory,shlyakhtenko1996random,helton2007operator}.

Wigner-type matrices is a special case for the  Kronecker random matrices introduced in  \cite{alt2017location}, and the  global law has been proved in Theorem 2.7 of \cite{alt2017location}, which states the following: let $H_n$ be  a Kronecker random matrix and $\mu_n^{H} $ be its empirical spectral distribution, then there exists a deterministic sequence of probability measure $\mu_n$ such that $\mu_n^{H}-\mu_n$ converges weakly in probability to the zero measure as $n\to\infty$.    In particular, for Wigner-type matrices,  the global law holds under the assumptions of  bounded variances and bounded moments. Our Theorem \ref{main} and Theorem \ref{main2} give a moment method proof of the global law in \cite{alt2017location} for Wigner-type matrices under bounded variances and Lindeberg's condition. Our new contribution is a  weaker condition for the convergence of the empirical spectral distribution $\mu_n^{M}$ of $M_n$.

In Section \ref{subsec:prof1} and Section \ref{subsec:prof2} we provide the proofs for Theorem \ref{main} and Theorem \ref{main2} respectively. We briefly summarize the proof ideas here. In the proof of Theorem \ref{main}, we revisit the standard path-counting moment method proof for the semicircle law (see for example \cite{bai2010spectral}). Since our matrix model has a variance profile, we encode different variances as weights on the paths and represent the moments of the empirical measure as a sum of homomorphism densities. Then if the tree homomorphism densities converge, the limiting spectral distribution exists. 
 
 For  the proof of Theorem \ref{main2}, since we assume that the variance profile  convergences under the cut norm, we can obtain a limiting graphon $W$. To obtain \eqref{QVE2} We  expand $a(z,x)$ in \eqref{QVE2} as a power series of homomorphism density from partially labeled trees to graphon $W$ denoted by $\beta_{2k}(x)$ in \eqref{DEF2}. Then we prove a graphon version of the Catalan number recursion formula for $\beta_{2k}(x)$ in \eqref{qve} and show that this essentially implies the quadratic vector equations \eqref{QVE1} and \eqref{QVE2}. This  recursion formula \eqref{qve} for tree homomorphism densities to a graphon could be of independent interest.

\subsection{Proof of Theorem \ref{main}} \label{subsec:prof1}
Using the truncation argument as in \cite{bai2010spectral,ding2014spectral}, we can first apply moment methods to a general Wigner-type matrix with bounded entries in the following lemma.

\begin{lemma}\label{prop}
Assume a Hermitian random matrix $A_n$ with a variance profile $S_n$ satisfies 
\begin{enumerate}
	\item  $\mathbb Ea_{ij}=0,\mathbb E|a_{ij}|^2=s_{ij}$. $\{a_{ij}\}_{1\leq i,j\leq n}$ are independent up to symmetry.
	\item $|a_{ij}|\leq \eta_n\sqrt n$ for some positive decreasing sequence $\eta_n$ such that $\eta_n\to 0$.
	\item $\sup_{ij}s_{ij}\leq C$ for a constant $C\geq 0$.
\end{enumerate}
Let $W_n$ be the graphon representation of $S_n$. Then for every fixed integer $k\geq 0$, we have the following asymptotic formulas:
\begin{align}\label{asymptotic}
 \frac{1}{n}\mathbb{E}[\textnormal{tr}M_n^{2k}]&=\sum_{j=1}^{C_k}t(T_j^{k+1},W_n)+o(1),\\
 \frac{1}{n}\mathbb{E}[\textnormal{tr}M_n^{2k+1}]&=o(1),\label{asymptotic2}
 \end{align}
 where $\{T_{j}^{k+1}, 1\leq j\leq C_k\}$ are all rooted planar trees of $k+1$ vertices.
 \end{lemma}
 
 \begin{proof}
 We start with expanding the expected normalized trace. For any integer $h\geq 0$,
\begin{align*}
\frac{1}{n}\mathbb{E}[\textnormal{tr}M_n^{h}]&=\frac 1 {n^{h/2+1}} \mathbb E\text{tr}(A_n^{h})
=\frac 1 {n^{h/2+1}}\sum_{1\leq i_1,\dots,i_{h}\leq n} \mathbb E [a_{i_1i_2}a_{i_2i_3}\cdots a_{i_{h}i_1}].
\end{align*}
Each term in the above sum corresponds to a closed walk (with possible self-loops) $(i_1,i_2,\dots, i_h)$ of length $h$ in the complete graph  $K_n$ on  vertices $\{1,\dots,n\}$.  Any closed walk can be classified into one of the following three categories.
\begin{itemize}
\item $\mathcal C_1$: All closed walks such that each edge appears exactly twice.
\item  $\mathcal C_2$: All closed walks that have at least one  edge  which appears only once.
\item  $\mathcal C_3$:  All other closed walks.

\end{itemize}
By independence, it's easy to see that every term corresponding to a walk in $\mathcal C_2$ is zero. We call a walk that is not in $\mathcal C_2$ a \textit{good walk}. Consider a good walk that uses $p$ different edges $e_1,\dots,e_p$
with corresponding multiplicity $t_1,\dots, t_p$ and each $t_i\geq 2$, such that $t_1+\cdots +t_p=h$. Now the term corresponding to a good walk has the form
$\mathbb E [a_{e_1}^{t_1}\cdots a_{e_p}^{t_p}].
$
Such a walk uses at most $p+1$ vertices and an upper bound for the number of good walks of this type is $n^{p+1}p^h$.
Since $|a_{ij} |\leq \eta_n\sqrt n$, and $\sup_{ij}\text{Var}(a_{ij})=\sup_{ij}s_{ij}\leq C$, we have
\begin{align*}\mathbb E a_{e_1}^{t_1}\cdots a_{e_p}^{t_p}\leq \mathbb  E[a_{e_1}^2]\cdots \mathbb E[a_{e_p}^2](\eta_n\sqrt n)^{t_1+\cdots+ t_p-2p}\leq \eta_n^{h-2p}n^{h/2-p}C^p. 
\end{align*}
When $h=2k+1$, we have
\begin{align*}
\frac{1}{n}\mathbb{E}[\textnormal{tr}M_n^{2k+1}]&=\frac 1 {n^{h/2+1}}\sum_{p=1}^k\sum_{\text{good walks of $p$ edges}} \mathbb E[a_{e_1}^{t_1}\cdots a_{e_p}^{t_p}]\\
&\leq \frac 1 {n^{k+3/2}}\sum_{p=1}^k n^{p+1}p^h(\eta_n^{h-2p}n^{h/2-p})C^p
\\
&=\sum_{p=1}^k p^h\eta_n^{h-2p}C^p=O(\eta_n)=o(1).
\end{align*}
When $h=2k$, let $S_i$ denote the sum of all terms in $\mathcal C_i, 1\leq i\leq 3$. By independence, we have $S_2=0$. 
Each walk in $\mathcal C_3$ uses $p$ different edges with $p\leq k-1$. We then have
\begin{align*}
S_3&=\frac 1 {n^{h/2+1}}\sum_{p=1}^{k-1}\sum_{\text{good walk of $p$ edges}} \mathbb Ea_{e_1}^{t_1}\cdots a_{e_p}^{t_p}\\
&\leq \frac 1 {n^{k+1}}\sum_{p=1}^{k-1} n^{p+1}p^h\left(\eta_n^{h-2p}n^{h/2-p}\right)\left(\sup_{ij}s_{ij}\right)^p
\\
&=\sum_{p=1}^{k-1}p^h\eta_n^{h-2p}C^p=o(1).
\end{align*}

Now it remains to compute $S_1$. For the closed walk that contains a self-loop, the number of distinct vertices is at most $k$, which implies the total contribution of such closed walks is  $O(n^k)$, hence  such terms are negligible in the limit of $S_1$. We only need to consider closed walks that use $k+1$ distinct vertices.
Each closed walk in $\mathcal C_1$ with $k+1$ distinct vertices in $\{1,\dots n\}$ is a closed walk on a tree of $k+1$ vertices that visits each edge twice.  
\begin{figure}
	\includegraphics[width=2 cm]{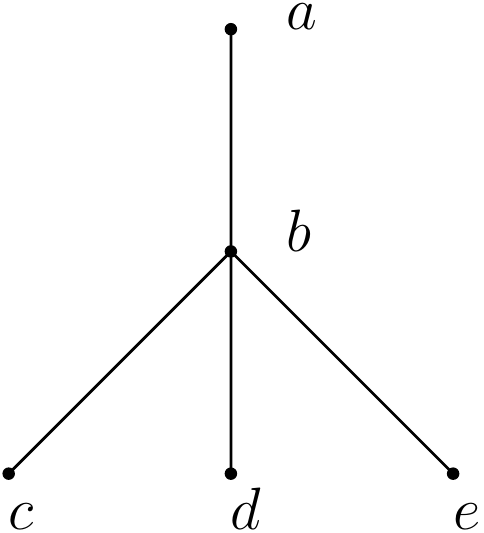}	
	\caption{ A closed walk $abcbdbeba$ corresponds to a  labeling of the rooted planar tree.}
	\label{label}
\end{figure}
Given an unlabeled rooted planar tree $T$ and
a depth-first search closed walk with vertices chosen from $[n]$, there is a one-to-one correspondence between such walk and a labeling of  $T$ (See Figure \ref{label}). There are $C_k$ many rooted  planar trees with $k+1$ vertices and for each rooted planar tree $T_{j}^{k+1}$, the ordering of the vertices from $1$ to $k+1$ is fixed by its depth-first search. Let $T_{l,j}^{k+1}$ be any labeled tree with the unlabeled rooted tree $T_j^{k+1}$ and a labeling $l=(l_1,\dots, l_{k+1}), 1\leq l_i\leq n, 1\leq i\leq k+1$ for its vertices from $1$ to $k+1$. For terms in $\mathcal{C}_1$, any possible labeling $l$ must satisfy that $l_1,\dots, l_{k+1}$ are distinct. 
Let $E(T_{l,j}^{k+1})$ be the edge set of $T_{l,j}^{k+1}$. Then $S_1$ can be written as
\begin{align}
S_1&=\frac 1 {n^{k+1}}\sum_{j=1}^{C_k}\sum_{l=(l_1,\dots,l_{k+1})}  \mathbb E \prod_{e\in E(T_{l,j}^{k+1})}a_{e}^{2} =\sum_{j=1}^{C_k}\frac{1}{n^{k+1}}\sum_{l=(l_1,\dots,l_{k+1})}  \prod_{e\in E(T_{l,j}^{k+1})}s_{e}.\label{316}
\end{align}
Consider 
$$
S_1':=\sum_{j=1}^{C_k}\frac{1}{n^{k+1}}\sum_{1\leq l_1,\dots,l_{k+1}\leq n}  \prod_{e\in E(T_{l,j}^{k+1})}s_{e}, 
$$
where  $l$ now stands for every possible labelling which allows  some of $l_1,\dots l_{k+1}$ to coincide, then we have 
\begin{align}
|S_1-S_1'|\leq \frac{1}{n^{k+1}}C_k(k+1)n^k(\sup_{ij}s_{ij}	)^k=O\left(\frac{1}{n}\right).\notag
\end{align}
On the other hand,
\begin{align}
t(T_j^{k+1},W_n)&=\int_{[0,1]^{k+1}} \prod_{uv\in E(T_{j}^{k+1})}W_n(x_u,x_v)dx_1 \dots dx_{k+1}\notag\\
&= \frac{1}{n^{k+1}}\sum_{1\leq l_1,\dots ,l_{k+1}\leq n} \prod_{uv\in E(T_{l,j}^{k+1})} s_{l_ul_v}=\frac{1}{n^{k+1}}\sum_{1\leq l_1,\dots, l_{k+1}\leq n}\prod_{e\in E(T_{l,j}^{k+1})} s_{e}.\label{321}
\end{align} 
Note that $\displaystyle S_1'=\sum_{j=1}^{C_k}t(T_{j}^{k+1},W_n).$
From \eqref{316} and \eqref{321}, we get 
$\displaystyle S_1=\sum_{j=1}^{C_k}t(T_j^{k+1},W_n)+o(1).$
Combining the estimates of $S_1,S_2$ and $S_3$, the conclusion of Lemma \ref{prop} follows.	 
 \end{proof}

 Lemma \ref{prop} connects the moments of the trace of $M_n$ to  homomorphism densities from trees to the graphon $W_n$. To proceed with the proof of Theorem \ref{main}, we need the following lemma.

 \begin{lemma}\label{sufficient}
	In order to prove the conclusion of Theorem \ref{main}, it suffices to prove it under the following conditions:
\begin{enumerate}
	\item $\mathbb Ea_{ij}=0$, $\mathbb E|a_{ij}|^2=s_{ij}$ and $\{a_{ij}\}_{1\leq i,j\leq n}$ are independent up to symmetry.
	\item $|a_{ij}|\leq \eta_n\sqrt n$ for some positive decreasing sequence $\eta_n$ such that $\eta_n\to 0$.
	\item $  \sup_{ij}s_{ij}\leq C$. \label{D'} for some constant $C\geq 0$.
\end{enumerate}
\end{lemma}

The proof of Lemma \ref{sufficient} follows verbatim as the proof of Theorem 2.9 in \cite{bai2010spectral}, so we do not give it here. The followings are two results that are used in the proof and will be used elsewhere in the paper, so we give them here. See Section A in \cite{bai2010spectral} for further details. 
\begin{lemma}[Rank Inequality] \label{rankinequality} Let $A_n, B_n$ be two $n\times n$ Hermitian matrices.  Let $F^{A_n}, F^{B_n}$ be the empirical spectral distributions of $A_n$ and $B_n$, then
	$$\|F^{A_n}-F^{B_n}\|\leq \frac{\textnormal{rank} (A_n-B_n)}{n},$$
	where $\|\cdot \|$ is the $L^{\infty}$-norm.
\end{lemma}
\begin{lemma}[L\'{e}vy Distance Bound]\label{levy} Let $L$ be the L\'{e}vy distance between two distribution functions, we have for any $n\times n$ Hermitian matrices $A_n$ and $B_n$,
	$$L^3(F^{A_n},F^{B_n})\leq \frac{1}{n} \textnormal{tr}[(A_n-B_n)(A_n-B_n)^*].$$
\end{lemma}
 With Lemma \ref{sufficient}, we will prove Theorem \ref{main} under assumptions in Lemma \ref{sufficient}.
\begin{proof}[Proof of Theorem \ref{main}] By Lemma \ref{sufficient}, it suffices to prove Theorem \ref{main} under the conditions (1)-(3) in Lemma \ref{sufficient}. We now assume these conditions hold. Then  \eqref{asymptotic} and \eqref{asymptotic2} in Lemma \ref{prop} can be applied here.

(1) Since for any finite tree $T$, $t(T,W_n)$ converges as $n\to\infty$, we can define 
\begin{align*}
	 \beta_{2k}&:=\lim_{n\to\infty}\frac{1}{n}\mathbb{E}[\textnormal{tr}M_n^{2k}]=\lim_{n\to\infty}\sum_{j=1}^{C_k}t(T_j^{k+1},W_n),\quad 
	 \beta_{2k+1}:=\lim_{n\to\infty}\frac{1}{n}\mathbb{E}[\textnormal{tr}M_n^{2k+1}]=0.
\end{align*}
With Carleman's Lemma (Lemma B.1 and Lemma  B.3 in \cite{bai2010spectral}), in order to to show the limiting spectral distribution of $M_n$ is uniquely determined by the moments,  it suffices to show that
 for each integer $k\geq 0$, almost surely we have 
	$$  \lim_{n\to\infty}\frac{1}{n}\text{tr}M_n^k= \beta_k,\quad \text{ and }\quad 
 \liminf_{k\to\infty}\frac{1}{k}\beta_{2k}^{1/2k}<\infty. 
 $$
 The remaining of the proof is similar to proof of Theorem 2.9 in \cite{bai2010spectral},  and we include it here for completeness.
Let $G(\mathbf{i})$ be the graph induced by the closed walk $\mathbf{i}=(i_1,\dots i_{k})$. Define $A(G(\textbf{i})):=a_{i_1i_2}a_{i_2i_3}\cdots a_{i_ki_1}$. Then
\begin{align}
\mathbb E\left| \frac{1}{n}\text{tr}M_n^{k}-\frac{1}{n}\mathbb{E}[\textnormal{tr}M_n^{k}] \right  |^4=\frac 1 {n^{4+2k}}\sum_{\mathbf{i}_j,1\leq j\leq 4}\mathbb E\prod_{j=1}^4[A(G(\mathbf{i}_j))-\mathbb EA(G(\mathbf{i}_j))]\notag
\end{align}

Consider a quadruple closed walk $\mathbf{i}_j, 1\leq j\leq 4$. By independence,
for the nonzero term, the graph $  \cup_{j=1}^4 G(\mathbf{i}_j)$ has at most two connected components. Assume there are $q$ edges  in $  \cup_{j=1}^4 G(\mathbf{i}_j)$ with multiplicity $v_1,\dots,v_q$, then  $v_1+\cdots+v_q=4k$. The number of vertices in $  \cup_{j=1}^4 G(\mathbf{i}_j)$ is at most $q+2$. To make every term in the expansion of $\mathbb E \prod_{j=1}^4\left(A(G(\mathbf{i}_j))-\mathbb EA(G(\mathbf{i}_j))\right) $ nonzero, the multiplicity of each edge is at least 2, so $q\leq 2k$ and the corresponding term satisfies
\begin{align}\label{AGbound}
\mathbb E\prod_{j=1}^4[A(G(\mathbf{i}_j))-\mathbb EA(G(\mathbf{i}_j))]\leq  {C^q} (\eta_n\sqrt n)^{4k-2q}.
\end{align}

If $q=2k$, we have $v_1=\cdots=v_q=2$. Since the graph  $  \cup_{j=1}^4 G(\mathbf{i}_j)$ has at most two connected components with at most $2k+1$ vertices,  there must be a cycle in  $  \cup_{j=1}^4 G(\mathbf{i}_j)$. So  the number of such graphs is at most $n^{2k+1}$.  Therefore from \eqref{AGbound},
 \begin{align*}
\mathbb E\left| \frac{1}{n}\text{tr}M_n^{k}-\frac{1}{n}\mathbb{E}[\textnormal{tr}M_n^{k}] \right  |^4&=\frac 1 {n^{4+2k}}\sum_{\mathbf{i}_j,1\leq j\leq 4}\mathbb E\prod_{j=1}^4[A(G(\mathbf{i}_j))-\mathbb EA(G(\mathbf{i}_j))]\\
&\leq \frac{1}{n^{4+2k}}\left(C^{2k}n^{2k+1}+\sum_{q<2k} C^qn^{q+2}(\eta_n\sqrt n)^{4k-2q}\right) =o\left(\frac 1 {n^2}\right).
\end{align*}
Then by Borel-Cantelli Lemma, \begin{align} \lim_{n\to\infty}\frac{1}{n}\text{tr}M_n^k= \beta_k \quad a.s.\notag\end{align} 
Moreover, since we have
$$ 
\beta_{2k}=\lim_{n\to\infty}\sum_{j=1}^{C_k}t(T_j^{k+1},W_n)\leq C_kC^k,$$
which implies
$\displaystyle
\liminf_{k\to\infty}\frac{1}{k}\beta_{2k}^{1/2k}=0.$

(2) Since $\delta_{\Box}(W_n,W)\to 0$, by Theorem \ref{graphonconvergence},
we have $$\lim_{n\to\infty} t(T_{j}^{k+1},W_n)= t(T_{j}^{k+1},W)$$ for any rooted planar tree $T_{j}^{k+1}$ with $k\geq 1, 1\leq j\leq C_k$. Therefore for all $k\geq 0$,
\begin{align*}
	\lim_{n\to\infty}\frac{1}{n}\textnormal{tr}M_n^{2k}&=\sum_{j=1}^{C_k}t(T_j^{k+1},W),\quad 
 \lim_{n\to\infty}\frac{1}{n}\textnormal{tr}M_n^{2k+1}=0 \quad a.s.
\end{align*}  
This completes the proof.
\end{proof}

\subsection{Proof of Theorem \ref{main2}} \label{subsec:prof2}
\begin{proof}
Since
$$  \limsup_{k\to\infty}(\beta_{2k}(x))^{1/(2k+1)}\leq 2\|W\|_{\infty}^{1/2}$$ 
 for all $x\in [0,1]$, we have for  $|z|>2\|W\|_{\infty}^{1/2}$,  $\displaystyle a(z,x)=\sum_{k=0}^{\infty}\frac{\beta_{2k}(x)}{z^{2k+1}}$ converges.  Note that
\begin{align*}
\int_{0}^{1}\beta_{2k}(x)dx&=\sum_{j=1}^{C_k}\int_{0}^{1} t_x(T_j^{k+1},W) dx=	\sum_{j=1}^{C_k}t(T_j^{k+1},W)=\beta_{2k},
\end{align*}
which implies for $|z|>2\|W\|_{\infty}^{1/2}$,
$\displaystyle
 s(z)=\sum_{k=0}^{\infty} \frac{\beta_{2k}}{z^{2k+1}}=\int_{0}^1 a(z,x) dx.	
$

Next we show \eqref{QVE2} holds for $|z|>2\|W\|_{\infty}^{1/2}$, which is equivalent to show
\begin{align}\label{qve}
a(z,x)\int_{0}^{1}W(x,y)a(z,y)dy=za(z,x)-1, \quad \forall x\in [0,1].	
\end{align}

We order the vertices in each rooted planar tree $T_{j}^{k+1}$ from $1$ to $k+1$ by depth-first search order (the root for each $T_{j}^{k+1}$ is always denoted  by $1$). Define a function
\begin{align*}
f_{j,k}(x_1,x_2,\dots, x_{k+1})=:\prod_{uv\in E(T_{j}^{k+1})}W(x_u,x_v).
\end{align*}
Now we expand $a(z,x)$ as follows
\begin{align*}
a(z,x)=&\sum_{k=0}^{\infty}\frac{1}{z^{2k+1}}\sum_{j=1}^{C_k} t_x (T_j^{k+1}, W)=	\sum_{k=0}^{\infty}\frac{1}{z^{2k+1}}\sum_{j=1}^{C_k} \int_{[0,1]^{k}}f_{j,k}(x,x_2,\dots, x_{k+1})\prod_{i=2}^{k+1}dx_i. 
\end{align*}
Then we can write $\displaystyle \int_{0}^{1} W(x,y)a(z,y)dy$ as
\begin{align}
\sum_{k=0}^{\infty}\frac{1}{z^{2k+1}}	\sum_{j=1}^{C_k} \int_{[0,1]^{k+1}} W(x,y)f_{j,k}(y,x_2,\dots, x_{k+1})dy\prod_{i=2}^{k+1} dx_i.
\label{38}
\end{align} 

Denote \begin{align}
B_{j,k}(x):=\int_{[0,1]^{k+1}} W(x,y)f_{j,k}(y,x_2,\dots, x_{k+1})dy\prod_{i=2}^{k+1} dx_i.\notag
\end{align}
Let $T_j^{k+1*}$ be the rooted  planar tree $T_{j}^{k+1}$ with a new edge attached to the root and the new vertex ordered $k+2$ (See Figure \ref{newlab}). Let $t_x(T_{j}^{k+1*},W) $ be the homomorphism density from partially labeled graph $T_{j}^{k+1*}$ to $W$ with the new vertex labeled $x$.
\begin{figure}
\includegraphics[width= 2.5 cm]{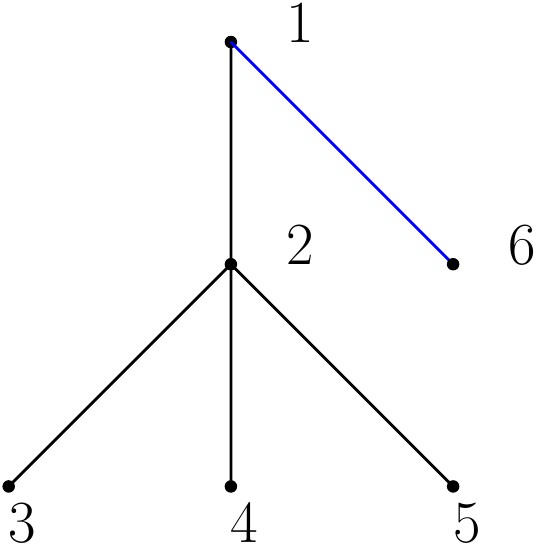}
\caption{A rooted planar tree with a new edge attached with a new vertex $6$}
\label{newlab}	
\end{figure}
With this notation, $B_{j,k}(x_{k+2})$ can be written as 
\begin{align}
&\int_{[0,1]^{k+1}} W(x_{k+2},x_1)f_{j,k}(x_1,x_2\dots, x_{k+1})\prod_{i=1}^{k+1} dx_i  \nonumber\\
=&\int_{[0,1]^{k+1}}\prod_{uv\in E(T_{j}^{k+1 *})}W(x_u,x_v) \prod_{i=1}^{k+1}dx_i
= t_{x_{k+2}}(T_{j}^{k+1*},W).\label{attach1}
\end{align}
So  \eqref{38} and \eqref{attach1} implies
$\displaystyle 
\int_{0}^{1} W(x,y)a(z,y)dy=	\sum_{k=0}^{\infty}\frac{1}{z^{2k+1}}\sum_{j=1}^{C_k} t_x(T_{j}^{k+1*},W).
$

Therefore
\begin{align}
a(z,x)\int_{0}^{1}W(x,y)a(z,y)dy 
=&	\left(\sum_{k=0}^{\infty}\frac{1}{z^{2k+1}}\sum_{i=1}^{C_k} t_x (T_i^{k+1}, W)\right)\left(\sum_{l=0}^{\infty}\frac{1}{z^{2l+1}}\sum_{j=1}^{C_l} t_x(T_{j}^{l+1*},W)\right)\notag\\
=&\sum_{k=0}^{\infty}\sum_{l=0}^{\infty}\frac{1}{z^{2(k+l)+2}}\sum_{i=1}^{C_k}\sum_{j=1}^{C_l}t_x(T_{i}^{k+1},W)t_x (T_j^{l+1*}, W)\label{startree}.
\end{align}

Let $\{T_{i,j}^{k+l+2}, 1\leq i\leq C_k, 1\leq j\leq C_l\} $ be all rooted planar trees with $k+l+2$ vertices generated by combining $T_i^{k+1}$ and $T_{j}^{l+1*}$ in the following way.
\begin{enumerate}
	\item First of all, by attaching the new labeled vertex of $T_j^{l+1*}$ to the root of $T_{i}^{k+1}$,  we get a new tree $T$ of $k+l+2$ vertices. 
	\item Choose the  root of $T$ to be the root of $T_{i}^{k+1}$. Order all vertices coming from $T_{i}^{k+1}$ with $1,2,\dots ,k+1$  and order vertices coming from $T_{j}^{l+1}$ with $k+2,k+3,\dots ,k+l+2$ both in depth-first search order. Then $T$ becomes a  rooted planar tree $T_{i,j}^{k+l+2}$ of $k+l+2$ vertices (See Figure \ref{newcom}).
\end{enumerate}

\begin{figure}
\includegraphics[width=9 cm]{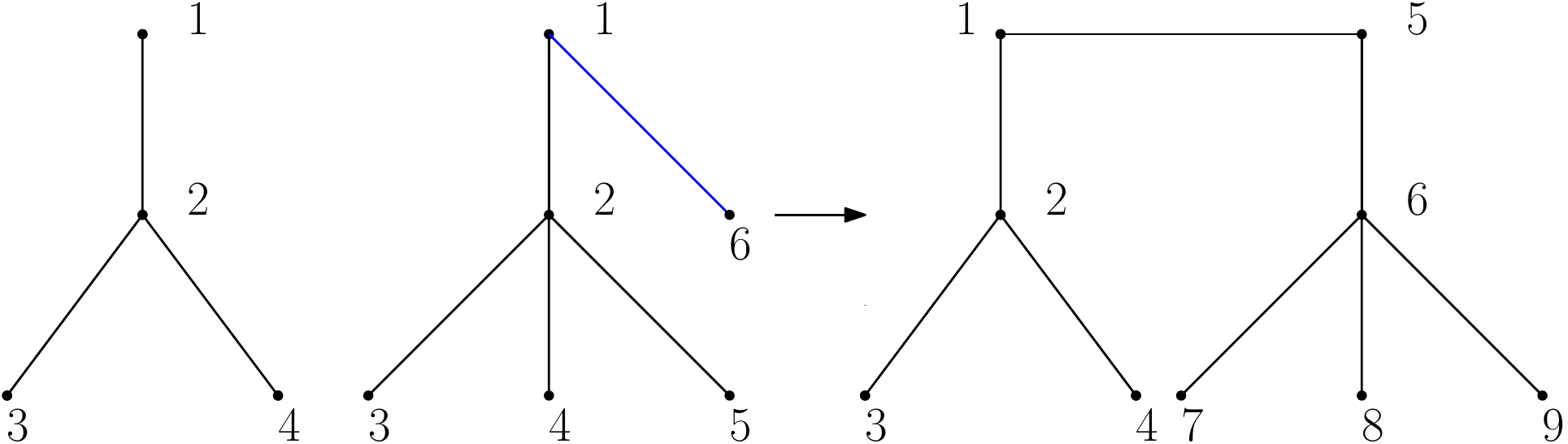}
\caption{Combining $T_{i}^{k+1}$ with $T_{j}^{l+1*}$ yields a new rooted planar tree of $k+l+2$ vertices.}
\label{newcom}	
\end{figure}
Let $t_x(T_{i,j}^{k+l+2},W)$ be the homomorphism density from partially labeled tree $T_{i,j}^{k+l+2}$ to $W$ with the root labeled $x$.  Using our notation, we have 
$$\displaystyle
	t_x(T_{i}^{k+1},W)t_x (T_j^{l+1*}, W)=t_x(T_{i,j}^{k+l+2},W)
.$$ Now let $s=k+l+1$, then \eqref{startree} can be written as 
\begin{align}
	\sum_{s=1}^{\infty}\frac{1}{z^{2s}}\sum_{\substack{ k+l+1=s\\k,l\geq 0}}\sum_{i=1}^{C_k}\sum_{j=1}^{C_l}t_x(T_{i,j}^{s+1},W).\label{catalan}
	\end{align}
Since all rooted planar trees in  the set $\{T_{i,j}^{s+1}$ $1\leq i\leq C_l, 1\leq j\leq C_k\}$ are different, from the Catalan number recurrence, there are 
$$\displaystyle \sum_{\substack{ k+l=s-1\\k,l\geq 0}}C_kC_l=\sum_{k=0}^{s-1}C_kC_{s-1-k}=C_{s}$$ 
many, which implies 	$\{T_{i,j}^{s+1}$ $1\leq i\leq C_l, 1\leq j\leq C_k\}$ are all rooted planar trees of $s+1$ vertices. Now \eqref{catalan} can be written as 
	\begin{align*}
	\sum_{s=1}^{\infty}\frac{1}{z^{2s}}\sum_{i=1}^{C_s} t_x(T_i^{s+1},W) 
	=za(z,x)-1. 
\end{align*}
 Therefore \eqref{qve}
 holds for $|z|>2\|W\|_{\infty}^{1/2}$. Since \eqref{qve} has a unique analytic solution  on $\mathbb C^+$ (see Theorem 2.1 in \cite{ajanki2015quadratic}), by analytic continuation, $a(z,x)$ has a unique extension on $\mathbb C^+\times [0,1]$ such that \eqref{qve} holds for all $z\in \mathbb C^+$. This completes the proof.
 \end{proof}
\section{Generalized Wigner Matrices} \label{generalized}

 The semicircle law for generalized Wigner matrices whose variance profile is doubly stochastic and comes from discretizing a function with zero-measure discontinuities  was proved in  \cite{nica2002operator,anderson2006clt}. The local semicircle law and universality of generalized Wigner matrices have been studied in \cite{erdHos2012bulk,erdHos2012rigidity} with a lower bound on the variance profile and conditions on the distributions of entries. 
 With Theorem \ref{main}, we can have a quick proof  of the semicircle law for generalized Wigner matrices under Lindeberg's condition. Compared to \cite{nica2002operator,anderson2006clt}, where the $L^{\infty}$-convergence of the variance profile is assumed,  we don't even need to assume the variance profile converges  under the cut metric. We will only need the weaker condition: the convergence of  $t(T,W_n)$ for any finite tree $T$.  In this section, we will show that the condition in Theorem \ref{main}, the convergence of tree integrals, is indeed a weaker condition than the convergence of the variance profile under the cut metric. Below we provide two examples where assumptions in \cite{anderson2006clt,shlyakhtenko1996random} fail, but our Theorem \ref{main} holds.

 We make the following assumptions for our \textit{generalized Wigner matrices}.  Let $A_n$ be a random Hermitian matrix such that entries are independent up to symmetry, and satisfies the following conditions:
\begin{enumerate}
	\item $\mathbb E[a_{ij}]=0,\mathbb E\left[|a_{ij}|^2\right]=s_{ij}$,\label{eq:condd1}
	\item $\displaystyle \frac{1}{n}\sum_{j=1}^n s_{ij}=1+o(1)$ for all $1\leq i\leq n$. \label{eq:cond2}
\item for any constant $\eta>0$,
$\displaystyle \lim_{n\to\infty} \frac{1}{n^2}\sum_{1\leq,i,j\leq n}\mathbb E\left[	|a_{ij}|^2\mathbf{1}(|a_{ij}|\geq \eta\sqrt n)\right]=0.$
\item 
$  \sup_{ij}s_{ij}\leq C $ for a constant $C>0$. \label{eq:condd4}
\end{enumerate}

We use our general formula in Theorem \ref{main} to get the semicircle law. An important observation is, when the variance profile is almost stochastic, the homomorphism densities in Theorem \ref{main} are easy to compute, as shown in the following lemma. The main idea is that we can start computing the homomorphism density integral from leaves on the tree.
\begin{lemma}\label{leaves}
Let $\{W_n\}_{n\geq 1}$ be any sequence of graphons such that $0\leq W_n(x,y)\leq C$ almost everywhere for some constant $C>0$. If for  $x\in [0,1]$ almost everywhere, $$\displaystyle  \lim_{n\to\infty}\int_0^1 W_n(x,y) dy=1,$$   then $\displaystyle \lim_{n\to\infty} t(T,W_n)=1$ for any finite tree $T$.
\end{lemma}
\begin{proof}
We induct on the number of vertices of a tree. Let $k=|V|$. 
For $k=2$, by Dominated Convergence Theorem, \begin{align}\label{eq:k=1}
 	\lim_{n\to\infty} t(T,W_n)=\int_{0}^1 W_n(x,y) dx dy=1.
 \end{align}

Assume for any trees with $k-1$ vertices the statement holds. For any tree $T$ with $k$ vertices, 
we order the vertices in $T$ by depth-first search. Then the vertex with label $k$  is a leaf. Note that
\begin{align*}
t(T,W_n)&=\int_{[0,1]^{k}} \prod_{ij\in E} W_n(x_i,x_j) dx_1\dots dx_{k}\notag\\&=\int_{[0,1]^{k}}	W_n(x_{k-1},x_{k})\prod_{ij\in E\setminus \{k-1,k\}} W_n(x_i,x_j) dx_1\dots dx_{k} \notag\\
&=\int_{[0,1]^{k-1}}	\left(\int_{[0,1]}W_n(x_{k-1},x_{k})dx_k\right)\prod_{ij\in E\setminus \{k-1,k\}} W_n(x_i,x_j) dx_1\dots dx_{k-1} 
\end{align*}
Let $T'$ be the tree $T$ with the edge $\{k-1,k\}$ removed, then we have 
\begin{align*}
t(T',W_n)&=\int_{[0,1]^{k-1}}\prod_{ij\in E\setminus \{k-1,k\}} W_n(x_i,x_j) dx_1\dots dx_{k-1},\\
 t(T,W_n)-t(T',W_n) &=\int_{[0,1]^{k-1}}	\left(\int_{[0,1]}W_n(x_{k-1},x_{k})dx_k-1\right)\prod_{ij\in E\setminus \{k-1,k\}} W_n(x_i,x_j) dx_1\dots dx_{k-1}. 
\end{align*}
By Dominated Convergence Theorem and \eqref{eq:k=1} we obtain
$$\lim_{n\to \infty}| t(T,W_n)-t(T',W_n)|=0.$$ 
Moreover, by our assumption of the induction,
$\displaystyle \lim_{n\to \infty}t(T',W_n)=1, $
therefore $\displaystyle \lim_{n\to\infty} t(T,W_n)= 1$. This completes the proof.
\end{proof}

Now we can give a quick proof of the semicircle law for generalized Wigner matrices in the following theorem, which is a quick consequence of Lemma \ref{leaves} and Theorem \ref{main}.  \begin{theorem}\label{gwm}
Let $A_n$ be a generalized Wigner matrix with assumptions above.
The limiting spectral distribution of $M_n:=\frac{A_n}{\sqrt n}$ converges weakly almost surely to the semicircle law.
\end{theorem}

\begin{proof} 	Let $W_n$ be the graphon representation of the variance profile for $A_n$. From Condition \eqref{eq:cond2}, we have \[\displaystyle \lim_{n\to\infty} \int_{[0,1]} W_n(x,y) dy=1\] for $x\in [0,1]$ almost everywhere. Then by Lemma \ref{leaves}, $\displaystyle \lim_{n\to\infty} t(T,W_n)=1$  for any finite tree $T$.

By part (1) in Theorem \ref{main}, the empirical spectral distribution of $M_n$ converges almost surely to a probability measure $\mu$ such that  for all $k\geq 0$.
\begin{align}\label{eq:deghomo}
\int x^{2k}d\mu=C_k, \quad \int x^{2k+1} d\mu=0.	
\end{align}
It's known that the semicircle law is uniquely determined by its moments, therefore the limiting spectral distribution for $M_n$ is the semicircle law.
\end{proof}

Theorem \ref{gwm} can be applied to study the spectrum of inhomogeneous random graphs with roughly equal expected degrees. This is a sparse random graph model where no limiting variance profile is assumed, so the theorems in \cite{shlyakhtenko1996random,anderson2006clt} do not apply here. Consider the  inhomogeneous Erd\H{o}s-R\'{e}nyi model $\mathcal G(n, (p_{ij}))$ with adjacency matrix $A_n$, where edges exist independently with given probabilities $p_{ij}$ such that $p_{ij}=p_{ji}$. Assume 
\begin{align}\label{eq:inhomo1}
	\sum_{i=1}^n p_{ij}=(1+o(1))n\alpha  \quad \text {for all } j\in [n]
\end{align} with some $\alpha\to 0,  \alpha=\omega\left(\frac{1}{n}\right)$, and 
\begin{align}\label{eq:inhomo2}
\displaystyle \max_{ij}p_{ij}\leq C\alpha \quad \text{ for some constant } C\geq 1.\end{align}
\begin{cor} Under the assumptions \eqref{eq:inhomo1} and \eqref{eq:inhomo2}, the empirical spectral distribution of the scaled adjacency matrix $\frac{A_n}{\sqrt{n\alpha}}$ converges almost surely to the semicircle law.	
\end{cor} 

\begin{proof}
	Consider the matrix $M_n=\frac{A_n-\mathbb EA_n}{\sqrt{\alpha}}$. Then by \eqref{eq:inhomo1} and \eqref{eq:inhomo2}, one can check that $M_n$ satisfies the assumptions \eqref{eq:condd1}-\eqref{eq:condd4} above for the generalized Wigner matrices. By Theorem \ref{gwm}, the empirical spectral distribution of $\frac{A_n-\mathbb EA_n}{\sqrt{n\alpha}}$ converges to the semicircle law almost surely. By Lemma \ref{levy}, we have almost surely
\begin{align} 
L^3\left(F^{\frac{A_n}{\sqrt {n \alpha }}}, F^{\frac{A_n-\mathbb EA_n}{\sqrt {n\alpha} }}\right)&\leq \frac{1}{n} \textnormal{tr}\left[\left(\frac{\mathbb EA_n}{\sqrt{n\alpha }}\right)^2 \right] =\frac{1}{n^2\alpha }	\sum_{i,j=1}^n  (\mathbb Ea_{ij})^2 \notag\\
&=\frac{\sum_{i,j=1}^n p_{ij}^2}{n^2\alpha}\leq \frac{n^2C^2\alpha^2 }{n^2\alpha}=C^2\alpha=o(1),\label{eq:sparseperturb}
\end{align}
where the last line of inequalities are from \eqref{eq:inhomo2}. Then $\frac{A_n}{\sqrt{n\alpha}}$ and $\frac{A_n-\mathbb EA_n}{\sqrt{n\alpha}}$ have the same limiting spectral distribution almost surely. This completes the proof.
\end{proof}

\section{Sparse $W$-random Graphs}\label{wrandomgraph}

Given a graphon $W:[0,1]^2\to [0,1]$, following the definitions in \cite{borgs2014p}, one can generate a sequence of sparse random graphs $G_n$ in the following way. We choose a sparsity parameter $\rho_n$ such that \[\sup_n\rho_n<1 \text{ with } \rho_n\to 0 \text{ and } n\rho_n\to\infty.\]
Let $x_1,\dots, x_n$ be i.i.d. chosen uniformly from $[0,1]$. For a graph $G_n$, $i$ and $j$ are connected with probability $\rho_n W(x_i,x_j)$ independently for all $i\not=j$. We define $G_n$ to be a \textit{sparse $W$-random graph}, and the sequence $\{G_n\}$ is denoted by $ \mathcal G(n,W,\rho_n)$. Note that we use the same i.i.d. sequence $x_1,\dots,x_n$ when constructing $G_n$ for different values of $n$ without resampling the $x_i$'s. We determine the limiting spectral distributions  for the adjacency matrices of sparse $W$-random graphs in the following theorem. This is a novel application of our theorem that cannot be covered by any previous results, since $W$ can be any  bounded measurable function.

\begin{theorem}\label{wrandom}
Let $G(n,W,\rho_n)$ be a sequence of sparse $W$-random graphs with adjacency matrices $\{A_n\}_{n\geq 1}$. The limiting spectral distribution of $\frac{A_n}{\sqrt {n\rho_n}}$ converges almost surely to a probability measure $\mu$ such that
\begin{align*}
\int_{\mathbb R} x^{2k} d\mu&=	\sum_{j=1}^{C_k}t(T_j^{k+1},W),\quad 
\int_{\mathbb R} x^{2k+1} d\mu=0.
\end{align*}
Moreover, its Stieltjes transform $s(z)$ satisfies the following equation:
\begin{align*}
s(z)&=\int_{0}^{1} a(z,x) dx, \quad 
a(z,x)^{-1}=z-\int_{0}^{1} W(x,y)a(z,y) dy, \quad \forall x\in [0,1].	  
\end{align*}
\end{theorem}
\begin{proof}

Let $$  B_n:=\frac{A_n-\mathbb E[A_n|x_1,\dots, x_n]}{\sqrt{\rho_n}}=(b_{ij})_{1\leq i,j\leq n}.$$ Note that $B_n$ is now a function of $x_1,\dots, x_n$. Since $n\rho_n\to\infty$ and $|b_{ij}|\leq \frac{2}{\sqrt{\rho_n}}$, 
, we have that for any constant $\eta>0$.
\begin{align}
\lim_{n\to\infty} \frac{1}{n^2}\sum_{1\leq,i,j\leq n}\mathbb E\left[	|b_{ij}|^2\mathbf{1}(|b_{ij}|\geq \eta\sqrt n)\mid x_1,\dots, x_n\right]=0,\notag\end{align} 
then the Lindeberg's condition \eqref{lind} holds for $B_n$.  Let $S_n$ be the variance profile matrix of $B_n$. Then we have $s_{ii}=0, 1\leq i\leq n$ and  for all $i\not=j$,
\begin{align}
s_{ij}=\frac{\rho_n W(x_i,x_j)(1-\rho_n W(x_i,x_j))}{\rho_n}=W(x_i,x_j)+o(1).	\notag
\end{align} 
Let  $W_n$ be the graphon representation of the matrix $S_n$ and let $\tilde{W}_n$ be the graphon of a weighted complete graph on $[n]$ with edge weights 
$W(x_i,x_j)$ for each edge $ij$. It implies that \begin{align}
W_n(x,y)=\Tilde{W}_n(x,y)+o(1),  \quad \forall (x,y)\in [0,1]^2.\notag	
\end{align} 
By Dominated Convergence Theorem, we get $\displaystyle
\lim_{n\to\infty}\delta_{\Box}(\Tilde{W}_n,W_n)= 0. 	
$ 
From Theorem 4.5 (a) in \cite{borgs2008convergent}, we have 
$\displaystyle
\lim_{n\to\infty}\delta_{\Box}(\Tilde{W}_n,W)= 0 	
$ almost surely, which implies   
$\displaystyle
\lim_{n\to\infty}\delta_{\Box}(W_n,W)= 0 	
$
almost surely.  Therefore from  Theorem \ref{main} (2),  the limiting spectral distribution of $\frac{B_n}{\sqrt n}$ exists almost surely and its moments and Stieltjes transform are given by Theorem \ref{main} and Theorem \ref{main2}. Next we show $\frac{B_n}{\sqrt n}$ and $\frac{A_n}{\sqrt{n\rho_n}}$ have the same limiting spectral distribution.

By Lemma \ref{levy}, we have almost surely
\begin{align}\label{eq:levybounds}
L^3(F^{\frac{A_n}{\sqrt {n \rho_n}}}, F^{\frac{B_n}{\sqrt n}})&\leq \frac{1}{n} \textnormal{tr}\left[\left(\frac{A_n}{\sqrt{n\rho_n}}-\frac{B_n}{\sqrt n}\right)^2 \right] =\frac{1}{n^2\rho_n}	\textnormal{tr} \left(\mathbb E[A_n| x_1,\dots, x_n]\right)^2.
\end{align}
By the way we generate our $W$-random graphs, we have for all $i\not=j$,
$$\mathbb E[(A_n)_{ij}\mid x_1,\dots, x_n]=\rho_n W(x_i,x_j).
$$
Therefore the right hand side in \eqref{eq:levybounds} is almost surely bounded by
\[\displaystyle \frac{\rho_n}{n^2}\sum_{i\not=j} W^2(x_i,x_j)\leq \rho_n=o(1),\]
which implies
$\displaystyle \lim_{n\to\infty}L^3(F^{\frac{A_n}{\sqrt {n \rho_n}}}, F^{\frac{B_n}{\sqrt n}})=0
$ almost surely. This completes the proof.
\end{proof}

\section{Random Block Matrices}\label{sec:randomblock}
Consider an $n\times n$ random Hermitian matrix $A_n$ composed of $d^2$ many rectangular blocks as follows.  We can write $A_n$ as
$  A_n:=\sum_{k,l=1}^d E_{kl}\otimes A_n^{(k,l)},	
$ where $\otimes$ denotes the Kronecker product of matrices, $E_{kl}$ are the elementary $d\times d$ matrices having $1$ at entry $(k,l)$ and $0$ otherwise. The blocks
$A_n^{(k,l)}, 1\leq k\leq l\leq d$ are of size $n_k\times n_l$ and consist of independent entries subject to symmetry. 
To summarize, we consider a \textit{random block matrix} $A_n$ with the following assumptions:
\begin{enumerate}\label{enu}
\item   $\displaystyle \lim_{n\to\infty}\frac{n_k}{n}= \alpha_k\in [0,1], 1\leq k\leq d$.\label{eq:blockassumption1}
\item $\mathbb Ea_{ij}=0, 1\leq i,j\leq n$, $\mathbb E|a_{ij}|^2=s_{kl}$ if $a_{ij}$ is in the $(k,l)$-th block. All entries are independent subject to symmetry.
\item  $  \sup_{kl}s_{kl}<C$ for some constant $C>0$. 
\item
 $\displaystyle \lim_{n\to\infty}\frac{1}{n^2}\sum_{ij} \mathbb E\left[(|a_{ij}|^2\mathbf{1}(|a_{ij}|\geq \eta\sqrt n)\right]=0, $ for any positive constant $\eta$.\label{eq:blockassumption4}
\end{enumerate}
For random block matrices with fixed $d$, the limiting spectral distributions are determined in \cite{far2008slow,ding2014spectral,avrachenkov2015spectral} under various assumptions. However, explicit moment formulas were not known. With Theorem \ref{main},   we can  compute the moments of the limiting spectral distribution. 
Let $W_n$ be the graphon  of the variance profile for $A_n$. Let 
$  \beta_0=0, \beta_i=\sum_{j=1}^i \alpha_j, i\geq 1.$ Then we can define the graphon $W$ such that
\begin{align}
W(x,y)=s_{kl},  \quad \text{if }(x,y)\in [\beta_{k-1}, \beta_k)\times [\beta_{l-1}, \beta_l)\label{gg}. 
\end{align}
Note that $W$ is a step function defined on $[0,1]^2$. Below is a version of Theorem \ref{main}, written specifically to address this model.

\begin{theorem}
Let $A_n$ be a random block matrix satisfying the assumptions above. Let $M_n=\frac{A_n}{\sqrt n}$ and $W$ be the graphon defined in \eqref{gg}. Then the limiting spectral distribution of $M_n$ converges almost surely to a probability measure $\mu$ such that
\begin{align}
\int_{\mathbb R} x^{2k} d\mu(x)&=	\sum_{j=1}^{C_k}t(T_j^{k+1},W), \quad 
\int_{\mathbb R} x^{2k+1} d\mu(x)=0, \label{eq2}
\end{align}
and its Stieltjes transform $s(z)$ satisfies
$\displaystyle
s(z)=\sum_{k=1}^d \alpha_k a_k(z),$
where for all $1\leq k\leq d$, $$\displaystyle a_k(z)^{-1}=z-\sum_{i=1}^d 
\alpha_i s_{ik}a_i(z).$$	 
 \begin{proof}
 From the definition, we have  $W_n(x,y)\to W(x,y)$  as $n\to\infty$ for  $(x,y)\in [0,1]^2$ almost everywhere. Hence
 \begin{align*}
 \|W_n-W\|_{\Box}&=\sup_{S,T\in [0,1]} \left|\int_{S\times T} W_n(x,y)-W(x,y) dx dy\right|\\
 &\leq \int_{[0,1]^2}|W_n(x,y)-W(x,y)| dxdy.
 \end{align*}
 Since $|W_n(x,y)|\leq C$, by the Dominated Convergence Theorem, we have
  $\displaystyle \|W_n-W\|_{\Box}\to 0$ as $n\to\infty$. \eqref{eq2} follow from Theorem \ref{main}.	The existence and uniqueness of $a_k(z), 1\leq k\leq d$ follows from Theorem 2.1 in \cite{ajanki2015quadratic}.
 \end{proof}
\end{theorem}

Now we consider the case where the number of blocks $d$ depends on $n$ such that $d\to\infty \text{ as } n\to\infty.$ 
We partition the $n$ vertices into $d$ classes: $
[n]=V_1\cup V_2\cup\cdots \cup V_d.	
$
Let $  m_0=0, m_i=\sum_{j=1}^{i} n_j$
and $$V_i = \{m_{i-1}+1, m_{i-1}+2,
\ldots, m_{i}\}$$ for $i= 1, \ldots, d$. 
We say the class $V_i$  is \textit{small} if $  \frac{n_i}{n}\to \alpha_i=0$, and $V_i$ is \textit{big} if $  \frac{n_i}{n}\to\alpha_i>0$.

 It's not necessary that $ \sum_{i=1}^{\infty}\alpha_i=1$. For example, if $n_i\leq  \log n$ for each $i$, we have $  \frac{n_i}{n}\to 0$ for all $i=1,2,\dots, $ then $  \sum_{i=1}^{\infty}\alpha_i=0$. In such case, a limiting graphon might not be well defined for general variance profiles. However, if we make all variances for the off-diagonal blocks to be $s_0$ for some constant $s_0$, then the limiting graphon will be a constant function $s_0$ on $[0,1]^2$  since all diagonal blocks will vanish to a zero measure set in the limit. 
 With these observations, we can extend our result to the case for $d\to\infty$ and $  \sum_{i=1}^{\infty}\alpha_i\leq 1$ under more assumptions on the variance profile.

\begin{theorem}\label{blockmatrix}
 Let $A_n$ be a random block matrix with  $d\to\infty$ as $n\to\infty$ satisfying assumptions \eqref{eq:blockassumption1}-\eqref{eq:blockassumption4}, then the  empirical spectral distribution of $\frac{A_n}{\sqrt{n}}$ converges almost surely to a probability measure $\mu$  if one of the extra conditions below holds.
 \begin{enumerate}
 \item 	$  \sum_{i=1}^{\infty}\alpha_i=1$ and $\alpha_1\geq \alpha_2\geq \cdots \geq 0$, or\label{eq:case1}
 \item $  \sum_{i=1}^{\infty}\alpha_i=\alpha <1$, $\alpha_1\geq \alpha_2\geq \cdots \geq 0$; also,  for any  two small classes $V_k, V_l, k\not=l$,   $s_{kl}=s_0$ for some constant $s_0$.
 	 For any large class $V_k$ and small class $V_l$,  $s_{kl}=s_{k0}$ for some constant $s_{k0}$.\label{eq:case2}
 \end{enumerate}
 
 \end{theorem}
  We illustrate the limiting graphon for case \eqref{eq:case2} in Figure \ref{infiblock}. Different colors represent different variances, and with our assumptions, all blocks of size $|V_k|\times |V_l|$ where $V_k,V_l$ are small converge to a diagonal line inside the last big block.
  \begin{figure}
 \includegraphics[width=2.5 cm]{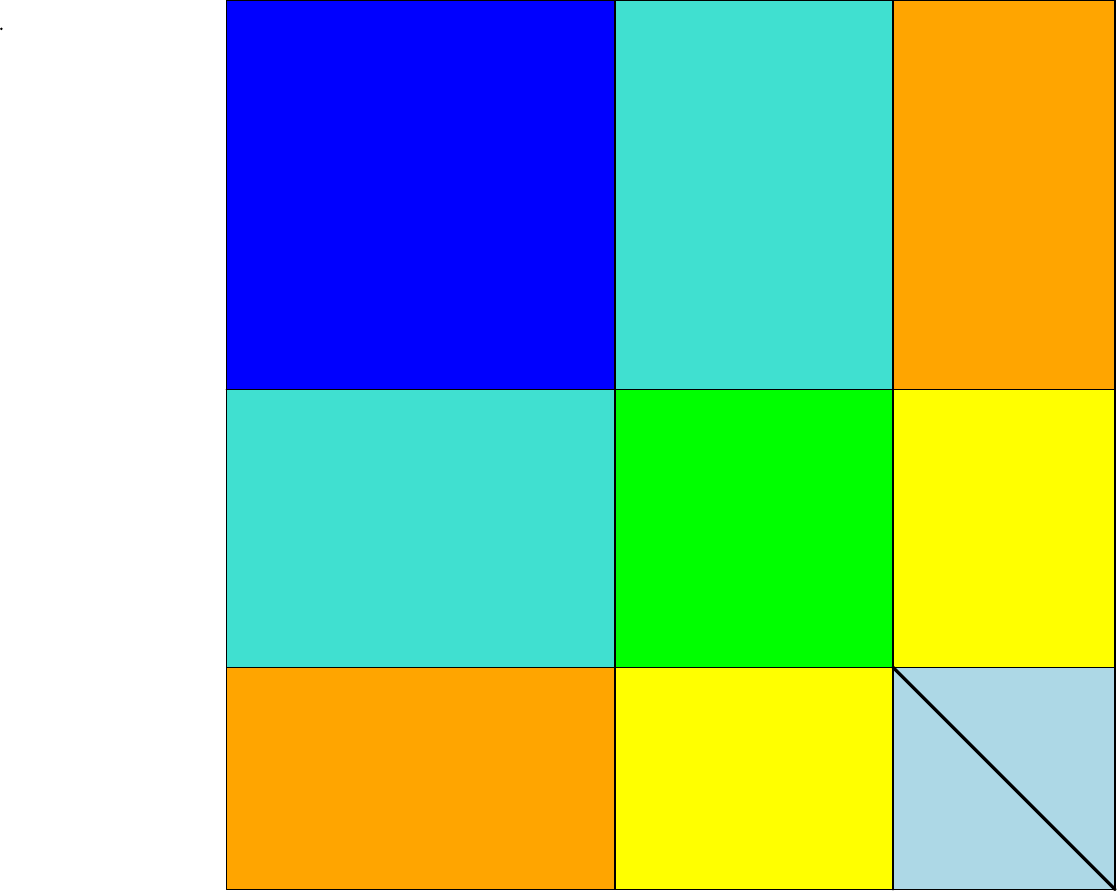}	
 \caption{limiting graphon with infinite many small classes}
 \label{infiblock}
 \end{figure}
 \begin{proof}[Proof of Theorem \ref{blockmatrix}]
 For case \eqref{eq:case1}, assume $  \sum_{i=1}^{\infty}\alpha_i=1$. Define 
$  \beta_0=0, \beta_i=\sum_{j=1}^i \alpha_j, i\geq 1.$ Then we can define a graphon $W$ as 
 \begin{align}
 W(x,y)= s_{ij}, \quad  \forall(x,y)\in[\beta_{i-1},\beta_i)\times [\beta_{j-1},\beta_j)\notag
 \end{align}
if $\beta_{i-1}\not= \beta_i,\beta_{j-1}\not=\beta_j$.
Then $W(x,y)$ is defined on $[0,1]^2$ almost everywhere. From our construction, $W_n(x,y)\to W(x,y)$ point-wise almost everywhere. By the Dominated Convergence Theorem, $\|W_n-W\|_{\Box}\to 0$.  For case \eqref{eq:case2}, similarly, we define $W$ in the following way,
 \begin{align*}
 W(x,y)=\left\{
 \begin{aligned}
 	&s_{ij},& \text{ if } &	(x,y)\in[\beta_{i-1},\beta_i)\times [\beta_{j-1},\beta_j), \alpha_i,\alpha_j\not=0,\\
 &s_0, &\text{ if } &(x,y)\in [\alpha,1]^2,\\
 &s_{i0}, &\text{ if }& (x,y)\in [\beta_{i-1},\beta_i)\times [\alpha,1] \text{ or } [\alpha,1]\times [\beta_{i-1},\beta_i).
 \end{aligned}
 \right.
 \end{align*}
Then $W$ is a graphon defined on $[0,1]^2$. Note that 
$ 
\lim_{n\to\infty} W_n(x,y)=W(x,y)	
$
for all $(x,y)\in [0,1]^2$ outside  the  subset of the diagonal $\displaystyle \{(x,y): x=y, x\in [\alpha,1]\},$
which is a zero measure set on $[0,1]^2$. So we have $\delta_{\Box}(W_n,W)\to 0$.  Then the result follows from Theorem \ref{main2}. 
 \end{proof}

 \section{Stochastic Block Models}\label{sSBM}

The adjacency matrix $A_n$ of a stochastic block model(SBM) with a growing number of classes is a random block matrix. A new issue here is $\mathbb EA_n\not=0$, which does not fit our assumptions in Section \ref{sec:randomblock}. However some perturbation analysis of the empirical measures can be applied to address this issue. In this section, we consider the  adjacency matrix $A_n$ for both sparse and dense SBMs with the following assumptions:
\begin{enumerate}
\item $  \frac{n_k}{n}\to \alpha_k\in [0,\infty), 1\leq k\leq d,$ where $d$ depends on $n$.
\item Diagonal elements in $A_n$ are $0$. Entries in the block $V_i\times V_i$ are independent Bernoulli random variables with parameter $p_{ii}$ depending on $n$ up to symmetry. Entries in the block $V_k\times V_l, k\not=l$ are independent Bernoulli random variables with parameter $p_{kl}$ depending on $n$. 
\item Let $  p=\sup_{ij} p_{ij}$. Assume $p=\omega(\frac{1}{n})$ and $  \sup_n p<1$.
\item Denote $\sigma^2:=p(1-p)$, and assume  \[\lim_{n\to\infty}\frac{p_{ij}(1-p_{ij})}{\sigma^2}=s_{ij}\in [0,1] \text{ for some constant } s_{ij}.\] 
\end{enumerate} 
    If $p\to 0$ (the sparse case), by the same argument in  \eqref{eq:sparseperturb}, $\frac{A_n-\mathbb EA_n}{\sigma\sqrt n}$ and $\frac{A_n}{\sigma\sqrt n}$ have the same limiting spectral distribution, we then have the following corollary from Theorem \ref{blockmatrix}.
    \begin{cor}
    	Let $A_n$ be the adjacency matrix of a sparse SBM with $p\to 0$,  $d\to\infty$ as $n\to\infty$. The empirical spectral distribution of $\frac{A_n}{\sqrt{n}}$ converges almost surely to a probability measure $\mu$  if one of the extra conditions below holds.
 \begin{enumerate}
 \item 	$  \sum_{i=1}^{\infty}\alpha_i=1$ and $\alpha_1\geq \alpha_2\geq \cdots \geq 0$, or
 \item $  \sum_{i=1}^{\infty}\alpha_i=\alpha <1$, $\alpha_1\geq \alpha_2\geq \cdots \geq 0$; also,  for any  two small classes $V_k, V_l, k\not=l$,   $s_{kl}=s_0$ for some constant $s_0$.
 	  For any large class $V_k$ and small class $V_l$,  $s_{kl}=s_{k0}$ for some constant $s_{k0}$.
 \end{enumerate}
    \end{cor}

   If $p\not\to 0$ (the dense case), to get the limiting spectral distribution of  the non-centered matrix $A_n$, we need to consider the effect of  $\mathbb EA_n$. If $\mathbb EA_n$ is of relatively low rank, we can still do a perturbation analysis from Lemma \ref{rankinequality}.
   The following theorem is a statement for the dense case.
\begin{cor}The empirical spectral distribution of the adjacent matrix $\frac{A_n}{\sqrt n \sigma}$ for a SBM with $p>c$ for a constant $c>0$ converges almost surely if $ d=o(n)$ and one of the following holds: \label{SBM}
\begin{enumerate}
 \item 	$  \sum_{i=1}^{\infty}\alpha_i=1$, $\alpha_1\geq \alpha_2\geq \cdots \geq 0$, or
 \item $ \sum_{i=1}^{\infty}\alpha_i=\alpha <1$, $\alpha_1\geq \alpha_2\geq \cdots \geq 0$. For any  two small classes $V_k, V_l, k\not=l$,   $s_{kl}=s_0$ for some constant $s_0$.
 	 For any large class $V_k$ and small class $V_l$,  $s_{kl}=s_{k0}$ for some constant $s_{k0}$.
 \end{enumerate} 
 \end{cor}
 
 \begin{proof} Let $\tilde{A}_n$ be a random block matrix such that
 $\tilde{a}_{ij}=a_{ij}$ for $i\not=j$ and $\{\tilde{a}_{ii}\}_{1\leq i\leq n}$ be independent Bernoulli random variables with parameter $p_{kk}$ if $i\in V_k$. Then $\textnormal{rank}(\mathbb E\tilde{A}_n)=d$. 
 
 Let $L\left(F^{\tilde{A}_n/\sigma\sqrt n},F^{{A}_n/\sigma\sqrt n}\right)$ be the L\'{e}vy distance between the empirical spectral measures of $\frac{A_n}{\sigma\sqrt n}$ and $\frac{\tilde{A}_n}{\sigma\sqrt n}$, then by Lemma \ref{levy}, 
 \begin{align}\label{eq:levyboundSBM}
 L^3\left(F^{\frac{\tilde{A}_n}{\sigma\sqrt n}},F^{\frac{{A}_n}{\sigma\sqrt n}}\right)	\leq \frac{1}{\sigma^2 n^2}\textnormal{tr}\left(\tilde{A}_n-A_n\right)^2=\frac{1}{\sigma^2n^2}\sum_{i=1}^n \tilde{a}_{ii}^2. 
 \end{align}
 The right hand side of \eqref{eq:levyboundSBM} is bounded by $\displaystyle \frac{1}{n\sigma^2}=o(1)$ almost surely. So we have almost surely 
 \begin{align}\label{eq:levybound}
 	\lim_{n\to\infty}L^3\left(F^{\frac{\tilde{A}_n}{\sigma\sqrt n}},F^{\frac{{A}_n}{\sigma\sqrt n}}\right)=0.
 \end{align}
Recall that  the limiting distribution of $\frac{\tilde{A}_n-\mathbb E\tilde{A}_n}{\sigma \sqrt n}$ exists from Theorem \ref{blockmatrix} for random block matrices.  By the Rank Inequality (Lemma \ref{rankinequality}), we have almost surely
 \begin{align}\label{eq:rankbound}
 	\left\|F^{\frac{\tilde{A}_n-\mathbb E\tilde{A}_n}{\sigma\sqrt n}}-F^{\frac{\tilde{A}_n}{\sigma \sqrt n}} \right\|\leq \frac{\textnormal{rank} (\tilde{A}_n-\mathbb E\tilde{A}_n-\tilde{A}_n)}{n}=\frac{\textnormal{rank} (\mathbb E\tilde{A}_n)}{n}=\frac{d}{n}=o(1).
\end{align}
	Then combining \eqref{eq:levybound} and \eqref{eq:rankbound}, almost surely $\frac{A_n}{\sigma\sqrt n}$ has the same limiting spectral distribution as $\frac{\tilde{A}_n-\mathbb E\tilde{A}_n}{\sigma\sqrt n}$.  The conclusion then  follows.
 \end{proof}

 Below, we give an example showing how to construct dense SBMs with a growing number of blocks which satisfies one of the assumptions in Corollary \ref{SBM}. Below is a lemma to justify that our two examples work.
 \begin{lemma} \label{sum1} Assume $  \sum_{i=1}^{\infty}\alpha_i=\alpha\leq 1$ and $1\geq\alpha_1\geq\alpha_2\geq \cdots>0$. Let $\displaystyle k(n):=\sup\left\{k: \alpha_k\geq \frac{1}{n} \right\},$ then $\displaystyle  \frac{k(n)}{n}=o(1)$.
\end{lemma}
\begin{proof} If not, there exists a subsequence $\{n_l\}$ such that $\frac{k(n_l)}{n_l}\geq \varepsilon>0$ for some $\varepsilon$.
Then  
\begin{align*}\frac{1}{n_l}\leq \alpha_{k(n_l)} \quad \text{ and }\quad \frac{k(n_l)-k(n_{l-1})}{n_l}\leq \sum_{i=k(n_{l-1})+1}^{k(n_l)}\alpha_i.	
\end{align*}
Hence
\begin{align}
\sum_{l=1}^{\infty} \frac{k(n_l)-k(n_{l-1})}{n_l}&\leq \sum_{i=1}^{\infty}\alpha_i=\alpha,\notag\\
\sum_{l=1}^{\infty}\frac{k(n_{l+1})-k(n_{l})}{k(n_{l+1})}&\leq \frac{\alpha}{\varepsilon}<\infty.\label{inf}
\end{align}
This implies $  \frac{k(n_{l+1})-k(n_{l})}{k(n_{l+1})}\to 0,$ so $\frac{k(n_{l+1})}{k({n_{l}})}\to 1$ as $n\to\infty$, therefore \eqref{inf} implies
\begin{align}\sum_{l=1}^{\infty}\frac{k(n_{l+1})-k(n_{l})}{k(n_{l})}<\infty.\label{finite}
\end{align}	
However, $$\displaystyle \sum_{l=1}^{\infty}\frac{k(n_{l+1})-k(n_{l})}{k(n_{l})}\geq \int_{k(n_1)}^{\infty}\frac{1}{x}dx=\infty,$$ 
which is a contradiction to \eqref{finite}. Lemma \ref{sum1} is then proved.
\end{proof}

\begin{example}
Let $\alpha_1\geq \alpha_2\geq \cdots>0$ and $\sum_{i=1}^{\infty}\alpha_i=1$. For each $n$, we generate the class $V_i$ with size $n_i=\lfloor n\alpha_i\rfloor$ for $i=1,2,\dots$ until $n_i=0$. Then we generate the last class $V_d$ with size $ n_d=n-\sum_{i=1}^{d-1}n_i$.	
Note that for every fixed $i$, $ \frac{n_i}{n}\to \alpha_i$. From Lemma \ref{sum1}, the number of blocks satisfies $d\leq k(n)+1=o(n).$
In particular, we have the following examples for the choice of $\alpha_i$'s:\begin{enumerate}
\item 	$  \alpha_i=\frac{C}{\gamma^i}$ for some constant $C,\gamma>0$ with $  \sum_{i=1}^{\infty}\alpha_i=1$.
\item $  \alpha_i=\frac{C}{i^{\beta}}$ for some $C>0,\beta>1$ with $  \sum_{i=1}^{\infty}\alpha_i=1.$
\end{enumerate}

 \end{example}

\begin{example}
	Let $  \alpha_1\geq \alpha_2\geq \cdots>0$ and $\sum_{i=1}^{\infty}\alpha_i=\alpha<1$. For each $n$, we can generate a class $V_i$ with size $n_i=\lfloor n\alpha_i\rfloor$ for $i=1,2,\dots,$ until $n_i=0$. Then generate $o(n)$ many small classes of size $o (n)$. By Lemma \ref{sum1}, $d=o(n)$.
\end{example}

\section{Random Gram Matrices}\label{gram}
In the last section, we present an example beyond general Wigner-type matrices to which  our main result can apply.
Let $X_n$ be a $m\times n$ complex random matrix whose entries  are independent. Consider a \textit{random Gram matrix} $M_n:=\frac{1}{n}X_nX_n^*$ with a variance profile matrix $S_n=(s_{ij})_{1\leq i\leq m, 1\leq j\leq n}$ satisfies the following conditions:
\begin{enumerate}
\item $\mathbb Ex_{ij}=0,\mathbb E|x_{ij}|^2=s_{ij}, \text{ for all }  1\leq i\leq m, 1\leq j\leq n. $
\item (Lindeberg's condition) for any constant $\eta>0$,\label{eta}
\begin{align}\lim_{n\to\infty} \frac{1}{nm}\sum_{i=1}^m\sum_{j=1}^n\mathbb E[	|x_{ij}|^2\mathbf{1}(|x_{ij}|\geq \eta\sqrt n)]=0.\label{lind}	
\end{align}
\item $  \sup_{ij}s_{ij}\leq C $ for some  constant $C\geq 0$.
\item $\displaystyle \lim_{n\to\infty}\frac{m}{n}=y\in(0,\infty)$.
\end{enumerate}

Let \begin{align}\label{eq:H_n}
 H_n:=\begin{bmatrix}
	0 & X_n\\
	X_n^* &0
\end{bmatrix}.\end{align} We first find the relation between the trace of $M_n$ and the trace of $H_n$ in the following lemma.

\begin{lemma} \label{trace} For any integer $k\geq 1$, the following holds:
\begin{align}
\frac{1}{m}\textnormal{tr}M_n^{k}=\frac{(m+n)^{k}}{2mn^k}\textnormal{tr}\left(\frac{H_n}{\sqrt{n+m}}\right)^{2k}.	
\end{align}
\end{lemma}
\begin{proof}
	It is a  simple linear algebra result that nonzero eigenvalues of $H$ come in pairs $\{-\sqrt{\lambda}, \sqrt{\lambda}\}$ where $\lambda>0$ is a non-zero eigenvalue of $X_nX_n^*$. Therefore for $k\geq 1$, 
	\begin{align}
	\textnormal{tr}(H_n^{2k})=2	\textnormal{tr}(X_nX_n^*)^k.
	\end{align}
We then have for $k\geq 1$,
\begin{align}
\frac{1}{m}\textnormal{tr}M_n^{k}&=\frac{1}{m}	\textnormal{tr}\left(\frac{1}{n}X_nX_n^*\right)^{k}=\frac{1}{2n^km}\cdot 2\textnormal{tr}(X_nX_n^*)^k=\frac{(m+n)^k}{2mn^k}\textnormal{tr}\left(\frac{H_n}{\sqrt{n+m}}\right)^{2k}.
\end{align}
\end{proof}
Since $H_n$ is a $(n+m)\times (n+m)$ general Wigner-type matrix with a variance profile 
\begin{align}\label{Sigma}
\Sigma_n:= \begin{bmatrix}
	0 & S_n\\
	S_n^T &0
\end{bmatrix},\end{align} we can decide the moments of the limiting spectral distribution of $M_n$ from Theorem \ref{main} and Lemma \ref{trace} in the following theorem.

 	\begin{theorem} \label{Wig} Let $M_n$ be a random Gram matrix with the assumptions above  and $W_n$ be the corresponding graphon of $\Sigma_n$. If for any finite tree $T$, $t(T,W_n)$ converges as $n\to\infty$, then the empirical spectral distribution of $M_n$ converges almost surely to a probability measure $\mu$ such that for $k\geq 1$, 
\begin{align*}
 \int x^{k} d\mu &=\frac{(1+y)^{k+1}}{2y}\sum_{j=1}^{C_k}\lim_{n\to\infty}t(T_j^{k+1},W_n).
 \end{align*} 
\end{theorem}

\begin{proof}
From Lemma \ref{trace}, for $k\geq 1$,
\begin{align}\label{85}
\frac{1}{m}\textnormal{tr}M_n^{k}&=\frac{(m+n)^{k+1}}{2mn^k}\cdot \frac{1}{n+m}\textnormal{tr}\left(\frac{H_n}{\sqrt{n+m}}\right)^{2k}.	
\end{align}
From Theorem \ref{main}, almost surely
$$
\lim_{n\to\infty}\frac{1}{n+m}\textnormal{tr}\left(\frac{H_n}{\sqrt{n+m}}\right)^{2k}=\sum_{j=1}^{C_k}\lim_{n\to\infty}t(T_j^{k+1},W_n).
$$
 Since $ \lim_{n\to\infty} \frac{m}{n}= y>0$,
The result follows from \eqref{85}. \end{proof}

Finally we derive the Stieltjes transform of the limiting spectral distribution from Theorem \ref{main2}.
\begin{theorem}\label{Wig2}
Let $M_n$ be a random Gram matrix with a variance profile $S_n$  and $W_n$ be the corresponding graphon of $\Sigma_n$ defined in \eqref{Sigma}. If $\delta_{\Box}(W_n,W)\to 0$ for some graphon $W$, then the  empirical spectral distribution of $  \frac{M_n}{\sqrt n}$ converges almost surely to a probability measure $\mu$ whose Stieltjes transform $s(z)$ is an analytic solution defined on $ \mathbb C^+$ by  the following equations:
	\begin{align} 
s(z)&=\frac{1+y}{y}\int_{0}^{\frac{y}{1+y}} b(z,u) du, \label{S87}\\
b(z,u)^{-1}&=z	-\int_{\frac{y}{1+y}}^1 \frac{W(u,v)}{(1+y)^{-1}-\int_{0}^{\frac{y}{1+y}}W(u,t)b(z,t)dt} dv,\label{S88}
\end{align}
where $b(z,u)$ is an analytic function defined on $  \mathbb C^+\times \left[0,\frac{y}{1+y}\right]$.
\end{theorem}

\begin{rmk}
Up to notational differences, \eqref{S87}, \eqref{S88}  are the centered case($\mathbb EM_n=0$) of the equations in \cite{hachem2006empirical} (see Section 5.1 in \cite{hachem2006empirical}), where a non-centered form of the equations were also derived under the assumptions of $(4+\varepsilon)$-bounded moments and the continuity of the variance profile. Recently,
\eqref{S87}, \eqref{S88} were also studied in \cite{alt2017local,alt2017}, where the local law for the centered case was proved under stronger assumptions including bounded $k$-moments of each entry for each $k$ and irreducibility condition on the variance profile. Our Theorem \ref{Wig} and Theorem \ref{Wig2} give the weakest assumption so far for the existence of the  limiting distribution and the quadratic vector equations only for the centered case. \end{rmk}

\begin{proof}
Let $s(z)$ be the Stieltjes transform of the limiting spectral distribution of $\frac{M_n}{\sqrt n}$. Let 
\[
\gamma_k:=	\int  x^k d\mu, \quad 
m_{2k}:=\sum_{j=1}^{C_k}t(T_j^{k+1},W),	\quad 
 \text{ and }  \quad 
	m(z):=\sum_{k=0}^{\infty}\frac{m_{2k}}{z^{2k+1}}.
\]
By Theorem \ref{Wig}, for $k\geq 1$,
$$\gamma_k=\frac{(1+y)^{k+1}}{2y}m_{2k}.
$$
Note that $m_0=\gamma_0=1$, we have for $|z|$ sufficiently large,
\begin{align}
s(z)&=\sum_{k=0}^{\infty}\frac{\gamma_k}{z^{k+1}}=	\frac{1}{z}+\sum_{k=1}^{\infty}\frac{m_{2k}}{z^{k+1}}\frac{1}{2y}(1+y)^{k+1}\notag\\
&=\sum_{k=0}^{\infty}\frac{m_{2k}}{z^{k+1}}\frac{1}{2y}(1+y)^{k+1}+\frac{y-1}{2yz}=\frac{1}{2y}\sqrt{\frac{1+y}{z}}m\left(\sqrt{\frac{z}{1+y}}\enskip \right)+\frac{y-1}{2yz}.\label{810}
\end{align}
From Theorem \ref{main} and \eqref{pse}, we know $m(z)$ is the Stieltjes transform of the limiting spectral distribution of $\frac{H_n}{\sqrt{n+m}}$. Moreover, from Theorem \ref{main2}, we have
\begin{align}
m(z)&=\int_{0}^1 a(z,u) du,\label{m}\\
a(z,u)^{-1}&=z-\int_{0}^1 W(u,v)a(z,v) dv,\label{a}
\end{align}
for some analytic function $a(z,u)$ defined on $\mathbb C^+\times [0,1]$. It remains to translate the equations above to an equation for $s(z)$. Let 
\begin{align*}
a_1(z,x):&=a(z,x), \text{ for } x\in \left[0,\frac{y}{1+y}\right],\\
a_2(z,x):&=a(z,x) , \text{ for }x\in \left[\frac{y}{1+y}, 1\right].	
 \end{align*}
 Since $\frac{m}{n}\to y\in (0,\infty)$, and $W_n$ is the corresponding graphon of $\Sigma_n$, its limit $W$  will have a bipartite structure, i.e., $W(u,v)=0$ for $  (u,v)\in \left[0,\frac{y}{1+y}\right]^2 \cup  \left[\frac{y}{1+y},1\right]^2$. Then we have the following equations from \eqref{a}:
\begin{align}
a_1(z,u)^{-1}&= z-\int_{\frac{y}{1+y}}^1 W(u,v)a_2(z,v) dv,\label{inv1}	\\
a_2(z,u)^{-1}&= z-\int_{0}^{\frac{y}{1+y}}W(u,v)a_1(z,v) dv.\label{inv2}
\end{align}
Combing \eqref{inv1} and \eqref{inv2}, we have the following self-consistent equation for $a_1(z,u)$:
\begin{align}
a_1(z,u)^{-1}&=z-\int_{\frac{y}{1+y}}^1 \frac{W(u,v)}{z-\int_{0}^{\frac{y}{1+y}}W(u,t)a_1(z,t) dt} dv.	\label{S888}
\end{align}
Let $\displaystyle
b(z,u):=\frac{a_1\left(\sqrt{\frac{z}{1+y}},u\right)}{\sqrt{z(1+y)}}.$
 Then $b(z,u)$ is an analytic function defined on $  \mathbb C^+\times \left[0,\frac{y}{1+y}\right]$.
 From \eqref{S888}, we can substitute $a_1(z,u)$ with $b(z,u)$ and get
\begin{align}
b(z,u)^{-1}&=z	-\int_{\frac{y}{1+y}}^1 \frac{W(u,v)}{(1+y)^{-1}-\int_{0}^{\frac{y}{1+y}}W(u,t)b(z,t)dt} dv.\label{b}
\end{align}
By multiplying with $a_1(z,u)$, $a_2(z,u)$ on both sides   in \eqref{inv1} and  \eqref{inv2} respectively, we have
\begin{align}
	1&=z a_1(z,u)-a_1(z,u)\int_{\frac{y}{1+y}}^1 W(u,v)a_2(z,v) dv,
	\label{816}\\
	1&=z a_2(z,u)-a_2(z,u)\int_{0}^{\frac{y}{1+y}} W(u,v)a_1(z,v) dv.\label{817}
\end{align}
From \eqref{816} and \eqref{817}, by integration with respect to $u$, we have
\begin{align*}
\frac{y}{1+y} &=z\int_{0}^{\frac{y}{1+y}}a_1(z,u)du-\int_{0}^{\frac{y}{1+y}}\int_{\frac{y}{1+y}}^1 W(u,v)a_1(z,u)a_2(z,v)dudv,\\	
\frac{1}{1+y} &=z\int_{\frac{y}{1+y}}^1 a_1(z,u)du-\int_{\frac{y}{1+y}}^1 \int_{0}^{\frac{y}{1+y}} W(u,v)a_2(z,u)a_1(z,v)dudv.
\end{align*}
Therefore we have 
\begin{align}
	\int_{0}^{\frac{y}{1+y}}a_1(z,u) du-\int_{\frac{y}{1+y}}^{1}a_2(z,u) du&=\frac{y-1}{z(1+y)}\label{818}.
\end{align}
From \eqref{m} and \eqref{818},  we have the following relation between $m(z)$ and $a_1(z,u)$:
\begin{align}
 m(z)&=\int_{0}^{\frac{y}{1+y}}a_1(z,u) du+	\int_{\frac{y}{1+y}}^{1}a_2(z,u) du=2\int_{0}^{\frac{y}{1+y}}a_1(z,u) du-\frac{y-1}{z(1+y)}.\label{eq:820}
\end{align}
 With \eqref{810}  and \eqref{eq:820}, we obtain the following equation for $s(z)$:
 \begin{align*}
 s(z)&=	\frac{1+y}{y}\int_{0}^{\frac{y}{1+y}} b(z,u)du,
 \end{align*} 
 where $b(z,u)$ satisfies the equation \eqref{b}. This completes the proof.

\end{proof}

\subsection*{Acknowledgement}  The author  thanks Ioana Dumitriu for introducing the problem, and Dimitri Shlyakhtenko, Roland Speicher for enlightening discussions. The author is grateful to L{\'a}szl{\'o} Erd{\H{o}}s for helpful comments on the first draft of this paper. 
\bibliographystyle{plain}
\bibliography{globalref}

 \end{document}